\documentclass[12pt]{amsart}
\usepackage{amsmath,amsthm}
\usepackage{amssymb}
\usepackage{graphics}
\usepackage{graphicx}
\usepackage[all]{xypic}
\SelectTips{cm}{12}
\UseTips{}
\usepackage{euscript}
\usepackage{amsfonts}
\usepackage{latexsym}
\usepackage{amssymb}
\usepackage{euscript}
\usepackage{epsf}
\usepackage{graphicx}
\theoremstyle{plain}
\swapnumbers
    \newtheorem{thm}{Theorem}[section]

    \newtheorem{prop}[thm]{Proposition}
    
    \newtheorem{lemma}[thm]{Lemma}
    
    \newtheorem{cor}[thm]{Corollary}

    \newtheorem{subsec}[thm]{}

\theoremstyle{definition}
    \newtheorem{defn}[thm]{Definition}
    \newtheorem{example}[thm]{Example}

\theoremstyle{remark}
        \newtheorem{remark}[thm]{Remark}

\newenvironment{myeq}[1][]
{\stepcounter{thm}\begin{equation}\tag{\thethm}{#1}}
{\end{equation}}

\newcommand{\mydiagram}[2][]
{\stepcounter{thm}\begin{equation}
     \tag{\thethm}{#1}\vcenter{\xymatrix{#2}}\end{equation}}
\newenvironment{mysubsection}[2][]
{\begin{subsec}\begin{upshape}\begin{bfseries}{#2.}
\end{bfseries}{#1}}
{\end{upshape}\end{subsec}}
\newenvironment{mysubsect}[2][]
{\begin{subsec}\begin{upshape}\begin{bfseries}{{#2}\vsn.}
\end{bfseries}{#1}}
{\end{upshape}\end{subsec}}
\newcommand{\supsect}[2]
{\vspace*{-5mm}\quad\\\begin{center}\textbf{{#1}}\vsm.~~~~\textbf{{#2}}\end{center}}
\newcommand{\w}[2][ ]{\ \ensuremath{#2}{#1}\ }
\newcommand{\ww}[2][ ]{\ \ensuremath{#2}{#1}}
\newcommand{\wh}{\ -- \ }
\newcommand{\wwh}{-- \ }
\newcommand{\wb}[2][ ]{\ (\ensuremath{#2}){#1}\ }
\newcommand{\wwb}[1]{\ (\ensuremath{#1})-}
\newcommand{\wref}[2][ ]{\ \eqref{#2}{#1}\ }
\newcommand{\hsp}{\hspace*{10 mm}}

\newcommand{\hsm}{\hspace*{2 mm}}

\newcommand{\vsm}{\vspace{2 mm}}
\newcommand{\vsn}{\vspace{1 mm}}
\newcommand{\hra}{\hookrightarrow}

\newcommand{\rest}[1]{\lvert_{#1}}

\newcommand{\epic}{\to\hspace{-5 mm}\to}

\newcommand{\clos}{\operatorname{cl}}

\newcommand{\diag}{\operatorname{diag}}
\newcommand{\Emb}{\operatorname{Emb}}

\newcommand{\eval}{\operatorname{ev}}
\newcommand{\oev}{\widehat{\eval}}
\newcommand{\Euc}[1]{\operatorname{Euc}^{#1}}

\newcommand{\fr}{\operatorname{fr}}

\newcommand{\Id}{\operatorname{Id}}
\newcommand{\Image}{\operatorname{Im}}
\newcommand{\Immers}{\operatorname{Imm}}
\newcommand{\imm}{\operatorname{im}}

\newcommand{\lk}{\operatorname{lk}}

\newcommand{\open}{\operatorname{op}}
\newcommand{\dpath}{d\sb{\operatorname{path}}}

\newcommand{\red}{\operatorname{re}}
\newcommand{\sign}{\operatorname{sgn}}
\newcommand{\SO}[1]{\operatorname{SO}({#1})}
\newcommand{\spn}{\operatorname{span}}
\newcommand{\Spin}[1]{\operatorname{Spin}({#1})}
\newcommand{\var}{\varepsilon}

\newcommand{\ba}{{\mathbf a}}
\newcommand{\bb}{{\mathbf b}}
\newcommand{\vb}{\vec{\bb}}

\newcommand{\be}{\mathbf{e}}
\newcommand{\ve}{\vec{\be}}
\newcommand{\vel}{\vec{\ell}}
\newcommand{\bbf}{\mathbf{f}}
\newcommand{\bg}{\mathbf{g}}
\newcommand{\bh}{\mathbf{h}}

\newcommand{\bn}{\mathbf{n}}
\newcommand{\vn}{\vec{\bn}}
\newcommand{\bp}{\mathbf{p}}
\newcommand{\vp}{\vec{\bp}}
\newcommand{\bq}{\mathbf{q}}
\newcommand{\vq}{\vec{\bq}}
\newcommand{\bS}[1]{{\mathbf S}\sp{#1}}

\newcommand{\bu}{\mathbf{u}}
\newcommand{\vu}{\vec{\bu}}
\newcommand{\bv}{\mathbf{v}}
\newcommand{\vv}{\vec{\bv}}
\newcommand{\bx}{\mathbf{x}}
\newcommand{\bbx}{\bar{\bx}}
\newcommand{\hx}{\hat{\bx}}
\newcommand{\vx}{\vec{\bx}}
\newcommand{\bw}{\mathbf{w}}
\newcommand{\vw}{\vec{\bw}}
\newcommand{\by}{\mathbf{y}}
\newcommand{\bz}{\mathbf{z}}
\newcommand{\vz}{\vec{\bz}}
\newcommand{\bze}{\mathbf{0}}
\newcommand{\vze}{\vec{\bze}}
\newcommand{\HH}{{\mathbb H}}
\newcommand{\NN}{{\mathbb N}}
\newcommand{\QQ}{{\mathbb Q}}
\newcommand{\R}{{\mathbb R}}
\newcommand{\RR}[1]{\R\sp{#1}}
\newcommand{\ZZ}{{\mathbb Z}}
\newcommand{\C}{{\EuScript C}}

\newcommand{\cE}{{\mathcal E}}
\newcommand{\hE}[1]{\widehat{\cE}\sb{#1}}
\newcommand{\F}{{\EuScript F}}
\newcommand{\hF}{\widetilde{\F}}
\newcommand{\G}{{\EuScript G}}
\newcommand{\htG}{\widetilde{\G}}
\newcommand{\hpi}{\widetilde{\pi}}
\newcommand{\hPhi}{\widetilde{\Phi}}
\newcommand{\eH}{{\mathcal H}}

\newcommand{\eL}{{\EuScript L}}
\newcommand{\eLn}[1]{\eL\sp{\times{#1}}}
\newcommand{\eLz}[1]{\eLn{#1}\sb{\bze}}

\newcommand{\nL}{{\EuScript N}}

\newcommand{\cP}{{\mathcal P}}
\newcommand{\T}{{\EuScript T}}

\newcommand{\Gn}[1]{\Gamma\sb{#1}}
\newcommand{\Gcl}[1]{\Gamma\sp{#1}\sb{\clos}}
\newcommand{\Gop}[1]{\Gamma\sp{#1}\sb{\open}}
\newcommand{\hGop}[1]{\widehat{\Gamma}\sp{#1}\sb{\open}}
\newcommand{\hC}{\widehat{\C}}
\newcommand{\tC}{\underline{\C}}
\newcommand{\Cs}{\hC\sb{\ast}}
\newcommand{\CG}{\hC(\Gamma)}
\newcommand{\tCG}{\tC(\Gamma)}

\newcommand{\CGn}[1]{\hC(\Gn{#1})}
\newcommand{\CsG}{\Cs(\Gamma)}
\newcommand{\CsGn}[1]{\Cs(\Gn{#1})}
\newcommand{\CsrG}{\Cs\sp{\red}(\Gamma)}
\newcommand{\CsrGn}[1]{\Cs\sp{\red}(\Gn{#1})}

\newcommand{\CiG}{\C\,\sp{\imm}(\Gamma)}

\newcommand{\CrGn}[1]{\hC\,\sp{\red}(\Gn{#1})}
\newcommand{\CsrD}{\Cs\sp{\red}(\Delta)}

\newcommand{\Eml}[2]{\Emb\sb{#1}({#2})}
\newcommand{\Ems}[1]{\Eml{\ast}{#1}}
\newcommand{\Emsr}[1]{\Emb\sp{\red}\sb{\ast}({#1})}
\newcommand{\Emr}[1]{\Emb\sp{\red}({#1})}
\newcommand{\hEmb}{\widehat{\Emb}}
\newcommand{\hEml}[2]{\hEmb\sb{#1}({#2})}
\newcommand{\hEms}[1]{\hEml{\ast}{#1}}
\newcommand{\hEmsr}[1]{\hEmb\sp{\red}\sb{\ast}({#1})}
\newcommand{\hEmr}[1]{\hEmb\sp{\red}({#1})}
\newcommand{\tEmb}{\underline{\Emb}}
\newcommand{\tpi}{\underline{\pi}}
\newcommand{\Na}[1]{\C({#1})}
\newcommand{\NaG}{\Na{\Gamma}}
\newcommand{\NaGn}[1]{\Na{\Gn{#1}}}
\newcommand{\Ns}[1]{\C\sb{\ast}({#1})}
\newcommand{\NsG}{\Ns{\Gamma}}

\newcommand{\Nr}[1]{\C\sp{\red}({#1})}
\newcommand{\NrGn}[1]{\Nr{\Gn{#1}}}
\newcommand{\Nsr}[1]{\C\sp{\red}\sb{\ast}({#1})}
\newcommand{\NsrG}{\Nsr{\Gamma}}
\newcommand{\NsrGn}[1]{\Nsr{\Gn{#1}}}

\newcommand{\oNsr}[1]{\overline{\C}\sp{\red}\sb{\ast}({#1})}
\newcommand{\oNsrG}{\oNsr{\Gamma}}
\newcommand{\TG}{\T\sb{\Gamma}}
\newcommand{\TGn}[1]{\T\sb{\Gn{#1}}}
\newcommand{\up}[1]{\sp{({#1})}}
\newcommand{\down}[1]{\sb{({#1})}}
\begin{document}
\title{Configuration spaces of spatial linkages: Taking Collisions Into Account }
\author{David Blanc}
\address{Dept.\ of Mathematics, U. Haifa, 3498838 Haifa, Israel}
\email{blanc@math.haifa.ac.il}
\author{Nir Shvalb}
\address{Dept.\ of Industrial Engineering, Dept.\ of Mechanical Engineering Ariel U., 4076113
Ariel, Israel}
\email{nirsh@ariel.ac.il}
\date{\today}
\begin{abstract}
	We construct a completed version \w{\CG} of the configuration space of a linkage
	$\Gamma$ in \w[,]{\RR{3}} which takes into account the ways one link can touch
	another. We also describe a simplified version \w{\tCG} which is a blow-up of the
	space of immersions of $\Gamma$ in \w[.]{\RR{3}} A number of simple detailed
	examples are given.
\end{abstract}
\maketitle

%
%
\section{Introduction}
\label{cint}

A \emph{linkage} is a collection of rigid bars, or links, attached to
each other at their vertices, with a variety of possible joints
(fixed, spherical, rotational, and so on). These play a central role in
the field of robotics, in both its mathematical and engineering aspects:
see \cite{MerlP,SeliG,TsaiR,FarbT}.

Such a linkage $\Gamma$, thought of as a metric graph, can be
embedded in an ambient Euclidean space \w{\RR{d}} in various ways,
called \emph{configurations} of $\Gamma$. The space \w{\NaG} of all
such configurations has a natural topology and differentiable
structure \wh see \cite{HallK} and Section \ref{ccs} below.

Such configuration spaces have been studied extensively, mostly for simple
closed or open chains (cf.\ \cite{FTYuzvT,GotRF,HKnutC,JSteiC,KMillM,MTrinG};
but see \cite{HolcM,KTsuC,OHaraM,SSB}). In  the plane, the convention is that
links freely slide over each other, so that a configuration is determined solely
by the locations of the joints. This convention is usually extended to spaces of
polygons in \w{\RR{3}} (see, e.g., \cite{KMillS}), so they no longer provide
realistic models of linkages (since in this model the links can pass through each
other freely). Alternatively, some authors have studies spaces of \emph{embeddings}
of sets of disjoint lines in \w[,]{\RR{3}} for which the issue does
not arise (cf.\ \cite{CEGSStolL,DViroC,ViroT}).

The goal of this paper is to address this question for spatial linkages, by
constructing a model taking into account how different bars touch each other.
We still use a simplified mathematical model, in which the links
have no thickness, and so on. Our starting point is the space \w{\NaG} of
\emph{embeddings} of $\Gamma$ in \w[.]{\RR{3}} If we allow the links to intersect,
we obtain the larger space \w{\CiG} of \emph{immersions} of $\Gamma$ in
\w[.]{\RR{3}} However, \w{\CiG} disregards the fact that in reality two bars touch
each other on one side or the other. To take this into account, we construct the
\emph{completed} configuration space \w{\CG} from \w{\NaG} by completing it 
with respect to a suitable metric. The new points of \w{\CG\setminus\NaG} are called
\emph{virtual configurations}: they correspond to immersed configurations decorated
with an additional (discrete) set of labels. See Section \ref{cccs}.

Unfortunately, the completed configuration space \w{\CG} is very difficult to
describe in most cases. Therefore, we also construct a simplified
version, called the \emph{blow-up}, denoted by \w[.]{\tCG} Here the
labelling is described explicitly by a finite set of invariants
(see Proposition \ref{pcovering} below), called \emph{linking numbers},
which determine the mutual position of two infinitely thin tangent cylinders
in space (cf.\ \cite{VassK}). See Section \ref{sbusc}.

One advantage of the blow-up is that its set of singularities can be filtered
in various ways, and the simpler types can be described explicitly. See
Section \ref{clblow}.

The relation between these spaces can be described as follows:
\mydiagram[\label{eqsimplconf}]{
\NaG ~\ar@{^{(}->}[rr] && \CG \ar@{->>}[rr] && \tCG \ar@{->>}[rr] && \CiG~.
}

The second half of the paper is devoted to the study of a number of examples:

\begin{enumerate}
\renewcommand{\labelenumi}{(\alph{enumi})~}
\item The simplest ``linkage'' we describe consists of two oriented lines
in \w[.]{\RR{3}} In this case \w[,]{\CG=\tCG}  and
the configuration space is described fully in \S \ref{cpgint}.\textbf{A}.
\item More generally, in Section \ref{cdpoints} we consider a collection of $n$ 
oriented lines in space, and show that its completed configuration space is 
homotopy equivalent to that of a linkage consisting of $n$ lines touching 
at the origin.

We give a full cell structure for \w{\CG} when \w{n=3} in Section \ref{ctml}.
\item Finally, the case of a closed quadrilateral chain is analyzed in
\S \ref{ccis}.\textbf{A} and that of an open chain of length three in
\S \ref{ccis}.\textbf{B}.
\end{enumerate}

%
%
\section{Configuration spaces}
\label{ccs}

Any embedding of a linkage in a (fixed) ambient Euclidean space \w{\RR{d}}
is determined by the positions of its vertices, but not all embeddings of its
vertices determine a legal embedding of the linkage. To make
this precise, we require the following:

\begin{defn}\label{dcspace}
An \emph{linkage type} is a graph
\w[,]{\TG=(V,E)} determined by a set $V$ of $N$ vertices and a set
\w{E\subseteq V^{2}} of $k$ edges (between distinct vertices). We assume there are
no isolated vertices.  A specific \emph{linkage} \w{\Gamma=(\TG,\vel)}  of type
\w{\TG} is determined by
a \emph{length vector} \w[,]{\vel:=(\ell_{1},\dotsc,\ell_{k})\in\RR{E}_{+}}
specifying the length \w{\ell_{i}>0} of each edge \w{(u_{i},v_{i})} in
$E$ \wb[.]{i=1,\dotsc,k} This \w{\vel} is required to satisfy the triangle
inequality where appropriate. We call an edge with a specified length a \emph{link},
(or bar) of the linkage $\Gamma$, and the vertices of $\Gamma$ are also known as
\emph{joints}.

An \emph{embedding} of \w{\TG} in the Euclidean space \w{\RR{d}} is an injective map
\w{\bx:V\to\RR{d}} such that the open intervals
\w{(\bx(u_{i}),\bx(v_{i}))} and \w{(\bx(u_{j}),\bx(v_{j}))} in \w{\RR{d}} are
disjoint if the edges \w{(u_{i},v_{i})} and \w{(u_{j},v_{j})} are distinct in $E$,
and the corresponding closed intervals \w{[\bx(u_{i}),\bx(v_{i})]} and
\w{[\bx(u_{j}),\bx(v_{j})]} intersect only at the common vertices.
The space of all such embeddings is denoted by \w[;]{\Emb(\TG)} it is an open
subset of \w[.]{(\RR{d})^{V}}

We have a \emph{moduli function} \w[,]{\lambda_{\TG}:(\RR{d})\sp{V}\to[0,\infty)}
written \w[,]{\bx\mapsto\lambda_{\TG}(\vx):E\to[0,\infty)} with
\w{(\lambda_{\TG}(\bx))(u_{i},v_{i}):=\|\bx(u_{i})-\bx(v_{i})\|}
for \w[.]{(u_{i},v_{i})\in E} We think of \w{\Lambda:=\Image(\lambda\sb{\TG})} as the
\emph{moduli space} for \w[.]{\TG}

The \emph{immersion space} of the linkage \w{\Gamma=(\TG,\vel)}
is the subspace \w{\CiG:=\lambda_{\TG}^{-1}(\vel)} of \w[.]{(\RR{3})\sp{V}}
A point \w{\bx\in\CiG} is called an \emph{immersed configuration} of $\Gamma$:
it is determined by the condition
\begin{myeq}\label{eqconfig}
\|\bx(u_{i})-\bx(v_{i})\|=\ell_{i}\hsp\text{for each edge}\hsm (u_{i},v_{i})\hsm
\text{in}~E~.
\end{myeq}

Finally, the \emph{configuration space} of the linkage
\w{\Gamma=(\TG,\vel)} is the subspace \w{\NaG:=\CiG\cap\Emb(\TG)}
of \w[.]{\Emb(\TG)} A point \w{\bx\in\NaG} is called an (embedded)
\emph{configuration} of $\Gamma$. Since \w{\Emb(\TG)} is open in
\w[,]{(\RR{d})\sp{V}} \w{\NaG} is open in \w[.]{\CiG}
\end{defn}

\begin{remark}\label{rlinkage}
We may also consider linkage types \w{\TG} containing lines (or half lines)
as ``generalized edges'' \w[:]{e\in E} in this case we add two (or one) new
vertices of $e$ to $V$, in order to ensure that any embedding of \w{\TG} in
\w{\RR{d}} is uniquely determined by the corresponding vertex embedding
\w[.]{\bx:V\to\RR{d}}
\end{remark}

\begin{example}\label{egchain}
The simplest kind of connected linkage is that of $k$-\emph{chain},
with $k$ edges (of lengths \w[),]{\ell_{1},\dotsc,\ell_{k}} in which all
nodes of degree $\leq 2$.

If all nodes are of degree $2$, it is called a \emph{closed} chain, and denoted
by \w[;]{\Gcl{k}} otherwise, it is an \emph{open} chain, denoted by \w[.]{\Gop{k}}
\end{example}

\begin{mysubsection}{Isometries acting on configuration spaces}\label{srestr}
The group \w{\Euc{d}} of isometries of the Euclidean space
\w{\RR{d}} acts on the spaces \w{\Emb(\TG)} and \w[,]{\NaG} and
the action is generally free, but certain configurations (e.g,
those contained in a proper subspace $W$ of \w[)]{\RR{d}} may be
fixed by certain transformations (e.g., those fixing $W$) (see \cite{KamiT}).

Note in particular that we may choose any fixed node \w{u_{0}} as
the \emph{base-point} of $\Gamma$, and the action of the
translation subgroup \w{T\cong\RR{d}} of \w{\Euc{d}} on \w{\bx(u_{0})}
\emph{is} free.  Therefore, the action
of $T$ on \w{\NaG} is also free. We call the quotient space
\w{\Ems{\TG}:=\Emb(\TG)/T} the \emph{pointed space of embeddings}
for \w[,]{\TG} and \w{\NsG:=\NaG/T} the \emph{pointed configuration space}
for $\Gamma$. Both quotient maps
have canonical sections, and in fact \w{\Emb(\TG)\cong\Ems{\TG}\times\RR{d}} and
\w[.]{\NaG\cong\NsG\times\RR{d}}
A pointed configuration (i.e., an element \w{[\bx]} of \w[)]{\NsG} is equivalent
to an ordinary configuration $\bx$ expressed in terms of a coordinate
frame for \w{\RR{d}} with \w{\bx(u_{0})=\bze} at the origin.

If we also choose a fixed edge \w{(u_{0},v_{0})} in \w{\TG}
starting at \w[,]{u_{0}} we obtain a smooth map
\w{p:\Emb(\TG)\to S^{d-1}} which assigns to a
configuration $\bx$ the direction of the vector from \w{\bx(u_{0})} to
\w[.]{\bx(v_{0})} The fiber \w{\Emsr{\TG}} of $p$ at \w{\ve_{1}\in S^{d-1}}
will be called the  \emph{reduced space of embeddings} of \w[,]{\TG} and the fiber
\w{\NsrG} of \w{p\rest{\NsG}} at \w{\ve\sb{1}} will be called the
\emph{reduced configuration space} of $\Gamma$. Note that the bundles
\w{\Emb(\TG)\to S\sp{d-1}} and \w{\NsG\to S\sp{d-1}} are locally trivial.
\end{mysubsection}

%
%
\section{Virtual configurations}
\label{cccs}

The space \w{\CiG} of immersed configurations can be used as a simplified model
for the space of all possible configurations of $\Gamma$.
However, this is not a very good approximation to the behavior of linkages in
$3$-dimensional space. We now provide a more realistic (though still simplified)
approach, as follows:

\begin{defn}\label{dpathemtr}
Note that since \w{j:\Emb(\TG)\to(\RR{3})\sp{V}} is an embedding into
a manifold, it has a \emph{path metric}: for any two functions
\w{\vx,\vx':V\to\RR{d}} we let \w{\delta'(\vx,\vx')} denote the infimum of the
lengths of the rectifiable paths from $\vx$ to \w{\vx'} in \w{\Emb(\TG)} (and
\w{\delta'(\vx,\vx'):=\infty} if there is no such path). We then let \
\w[.]{\dpath(\vx,\vx'):=\min\{\delta'(\vx,\vx'),1\}}
This is clearly a metric, which is topologically equivalent to the Euclidean
metric on \w{\Emb(\TG)} inherited from \w{(\RR{3})\sp{V}}
(cf.\ \cite[Lemma 6.2]{LeeRM}).  The same is true of the metric \w{\dpath}
restricted to the subspace \w{\NaG} (compare \cite{RRimoI}).
\end{defn}

\begin{remark}\label{rsmooth}
Since any continuous path \w{\gamma:[0,1]\to\Emb(\TG)} has an
$\epsilon$-neighborhood of its image still contained in \w[,]{\Emb(\TG)\subseteq(\RR{d})\sp{V}} by the Stone-Weierstrass Theorem
(cf.\ \cite[\S 3.7]{FrieM}) we can approximate $\gamma$ by a smooth
(even polynomial) path $\hat{\gamma}$. Thus we may assume that all paths
between configurations used to define \w{\dpath} are in fact smooth.
\end{remark}

\begin{defn}\label{dcemb}
We define the \emph{completed space of embeddings} of \w{\TG} to be the
completion \w{\hEmb(\TG)} of \w{\Emb(\TG)} with respect to the metric \w{\dpath}
(cf.\ \cite[Theorem 43.7]{MunkrTF}). The \emph{completed configuration space}
\w{\CG} of a linkage $\Gamma$ is similarly defined to be the completion of the
embedding configuration spaces \w{\NaG} with respect to \w[.]{\dpath\rest{\NaG}}
The new points in \w{\CG\setminus\NaG} will be called
\emph{virtual configurations}: they correspond to actual immersions of $\Gamma$
in \w{\RR{d}} in which (infinitely thin) links are allowed to touch,
``remembering'' on which side this happens.

The spaces we have defined so far fit into a commutative diagram as follows:

\mydiagram[\label{eqrelconfsp}]{
\NsrG \ar@{^{(}->}[r] \ar@{^{(}->}[d] & \NsG \ar@{^{(}->}[r] \ar@{^{(}->}[d] &
\NaG \ar@{^{(}->}[r] \ar@{^{(}->}[d] \ar@/_1.3pc/[l] & \CG \ar@{->>}[r] \ar[d] &
\CiG \ar@{^{(}->}[d] \\
\Emsr{\TG} \ar@{^{(}->}[r] & \Ems{\TG} \ar@{^{(}->}[r] &
\Emb(\TG) \ar@{^{(}->}[r] \ar@/^1.0pc/[l] & \hEmb(\TG) \ar@{->>}[r] &
(\RR{3})\sp{V}.
}

\vsm\quad

The pointed and reduced versions \w[,]{\hEms{\TG}} \w[,]{\hEmsr{\TG}} \w[,]{\CsG)}
and \w[,]{\CsrG} are defined as in \S \ref{srestr}, and fit into a suitable
extension of \wref[.]{eqrelconfsp}

Note that the moduli function \w{\lambda_{\TG}:\Emb(V)\to\RR{E}} of \S \ref{dcspace}
extends to \w[,]{\hat{\lambda}\sb{\TG}:\hEmb(\TG)\to\RR{E}} and in fact \w{\CG} is
just the pre-image \w{\hat{\lambda}\sb{\TG}(\vel)} for the appropriate vector of
lengths $\vel$.
\end{defn}

\begin{remark}\label{rcompletion}
Even though the metric \w{\dpath} is topologically equivalent to the Euclidean
metric \w{d_{2}} on \w[,]{\Emb(\TG)} its completion with respect to the latter
is simply \w[,]{(\RR{d})^{V}} so the corresponding completion of \w{\NaG} is
the space of immersed configurations \w{\CiG}  of \S \ref{dcspace}. In fact,
from the properties of the completion we deduce:
\end{remark}

\begin{lemma}\label{lgimmerse}
For any linkage $\Gamma$ of type \w[,]{\TG} there is a continuous map
\w[.]{q:\CG\to\CiG}
\end{lemma}

\begin{remark}\label{rquotient}
The map $q$ is a quotient map, as long as \w{\NaG} is dense in
\w[.]{\CiG}  This may fail to hold if $\Gamma$ has rigid non-embedded 
immersed configurations, which are isolated points in \w[).]{\CiG}
In such cases we add the associated ``virtual completed configurations'' as
isolated points in \w[,]{\CG} so that \w{q':\CG\to\CiG}  extends to a surjection.
\end{remark}

%
%
\section{Blow-up of singular configurations}
\label{sbusc}

As in the proof of Lemma \ref{lgerm}, we may think of points in \w{\hEmb(\TG)} as
Cauchy sequences \w{(x_{i})_{i=1}^{\infty}=(\gamma(t_{i}))_{i=1}^{\infty}} lying on a
smooth path $\gamma$ in \w[,]{\Emb(\TG)} which we can partition into  equivalence
classes according to the limiting tangent direction
\w{\vv:=\lim\sb{i\to\infty}\gamma'(t\sb{i})} of $\gamma$.

We can use this idea to construct an approximation to \w[,]{\hEmb(\TGn{n})}
by blowing up the singular configurations where links or joints of
$\Gamma$ meet (cf.\ \cite[II, \S 4]{ShafB1}). For our purpose the following
simplified version will suffice:

\begin{defn}\label{dlinknum}
Given an abstract linkage \w[,]{\TG=(V,E)} we have an orientation for each
\emph{generalized edge} (that is, edge, half-line, or line) \w[,]{e\in E}
by Remark \ref{rlinkage}.
Let $\cP$ denote the collection of all ordered pairs \w{(e',e'')\in E\sp{2}}
of distinct edges of $\Gamma$ which have no vertex in common.

For each embedding \w{\bx:V\to\RR{3}} of \w{\TG} in space and each pair
\w[,]{\xi:=(e',e'')\in\cP} let $\vw$ denote the \emph{linking vector}
connecting the closest points \w{\ba\in\bx(e')} and \w{\bb\in\bx(e'')}
on the generalized segments
\w{\bx(e')} and \w{\bx(e'')} (in that order). Since $\bx$ is an embedding,
\w[.]{\vw\neq\vze} We define an invariant
\w{\phi\sb{\xi}(\bx)\in\ZZ/3=\{-1,0,1\}} by:
$$
\phi\sb{\xi}(\bx)~:=~\begin{cases}
\sign(\vw\cdot(\bx(e')\times\bx(e''))) & \text{if $\ba$ is interior
to\w[,]{\bx(e')} $\bb$ to \w[,]{\bx(e'')}}\\
& \text{and $\bx(e')$ and $\bx(e'')$ are not coplanar}\\
0 & \text{otherwise.}
\end{cases}
$$
\noindent This is just the \emph{linking number} \w{\lk\sb{(\ell',\ell'')}}
of the lines \w{\ell'} and \w{\ell''} containing \w{\bx(e')} and
\w[,]{\bx(e'')} respectively, although the usual convention is that
\w{\lk\sb{(\ell',\ell'')}} is undefined when the two lines are coplanar
(cf.\ \cite{DViroC}).
\end{defn}

\begin{defn}\label{dblowupr}
If we let $\F$ denote the product space \w[,]{(\RR{3})\sp{V}\times(\ZZ/3)\sp{\cP}}
the collection of invariants \w{\phi\sb{\xi}(\bx)} together define a
(not necessarily continuous) function \w[,]{\Phi:\Emb(\TG)\to\F} equipped
with a projection \w[,]{\pi:\F\to(\RR{3})\sp{V}} such that \w{\pi\circ\Phi} is the
inclusion \w{j:\Emb(\TG)\hra(\RR{3})\sp{V}} of \S \ref{dpathemtr}.

We now define an equivalence relation $\sim$ on $\F$ generated as follows:
consider a Cauchy sequence \w{(\bx\sb{i})\sb{i=1}\sp{\infty}} in
\w{X=\Emb(\TG)} with respect to the path metric \w{\dpath}
(cf.\ \S \ref{dpathemtr}).
Since \w{j:\Emb(\TG)\hra(\RR{3})\sp{V}} is an inclusion into a complete
metric space, and \w{\dpath} bounds the Euclidean metric in \w[,]{(\RR{3})\sp{V}}
the sequence \w{j(\bx\sb{i})\sb{i=1}\sp{\infty}} converges to a point
\w[.]{\bx\in(\RR{3})\sp{V}}

If there are two (distinct) sequences
\w{\vec{\alpha}=(\alpha\sb{\xi})\sb{\xi\in\cP}}
and \w{\vec{\beta}=(\beta\sb{\xi})\sb{\xi\in\cP}} in \w{(\ZZ/3)\sp{\cP}}
and a Cauchy sequence \w{(\bx\sb{i})\sb{i=1}\sp{\infty}} as above such that for
each \w{N>0} there are \w{m,n\geq N} with \w{\phi\sb{\xi}(\bx\sb{m})=\alpha\sb{\xi}}
and \w{\phi\sb{\xi}(\bx\sb{n})=\beta\sb{\xi}} for all \w[,]{\xi\in\cP} then
we set \w{(\bx,\vec{\alpha})\sim(\bx,\vec{\beta})} in $\F$, where
\w{\bx=\lim\sb{i}\ j(\bx\sb{i})} in \w[.]{(\RR{3})\sp{V}}

Finally, let \w{\hF:=\F/\sim} be the quotient space, with \w{q:\F\to\hF} the
quotient map. Note that the projection \w{\pi:\F\to(\RR{3})\sp{V}} induces a
well-defined surjection \w[.]{\hpi:\hF\to(\RR{3})\sp{V}} The other
projection \w{\F\to(\ZZ/3)\sp{\cP}} induces the \emph{completed linking number}
invariants \w{\phi\sb{\xi}\colon\hF\to\ZZ/3} for each \w{\xi\in\cP} (where
\w{\phi\sb{\xi}(\bx)} is set equal to $0$ if \w[).]{\alpha\sb{\xi}\neq\beta\sb{\xi}}
\end{defn}

By Definition \ref{dcemb} we have the following:

%
%
\begin{prop}\label{pcontinuous}
The function \w{\Phi:\Emb(\TG)\to\F} induces a continuous map
\w[.]{\hPhi:\hEmb(\TG)\to \hF}
\end{prop}

\begin{defn}\label{dblowup}
Given an abstract linkage \w[,]{\TG=(V,E)} the image of the
map \w[,]{\hPhi:\hEmb(\TG)\to \hF} denoted by \w[,]{\tEmb(\TG)} is called the
\emph{blow-up} of the space of embeddings \w[.]{\Emb(\TG)}  It contains the
\emph{blow-up} \w{\tCG} of the configuration space  \w{\CG}  as a closed subspace;
this is defined to be the closure of the image of \w[.]{\hPhi\rest{\CG}}
\end{defn}

\begin{mysubsect}{The singular set of \ $\CiG$}\label{sfilt}
It is hard to analyze the completed space of embeddings \w{\hEmb(\TG)} or
the corresponding configuration space \w[,]{\CG} since the new
points are only describable in terms of Cauchy sequences in \w[.]{\Emb(\TG)}
However, for most linkages $\Gamma$, the space \w{\NaG} of embedded configurations is
dense in the space of immersed configurations \w[.]{\CiG} We denote its complement
by \w[.]{\Sigma:=\CiG\setminus\NaG} A point (graph immersion) \w{\bx\in\Sigma}
must have at least one intersection between edges not at a common vertex.

Note that generically, in a dense open subset $U$ of $\Sigma$,
for any two edges \w{(e',e'')\in\cP} the intersection of \w{\bx(e')} and \w{\bx(e'')}
is at an (isolated) points internal to both edges, and each of
the intersections are independent.
\end{mysubsect}

%
%
\begin{prop}\label{pcovering}
All fibers of $\hpi$ are finite, the restriction of $\tpi$ to \w{\NaG} (or
\w[)]{\Emb(\TG)} is an embedding, while the restriction of
$\tpi$ to \w{\tpi\sp{-1}(U)\to U} is a covering map.
\end{prop}

\begin{proof}
Observe that the identifications made by the equivalence relation $\sim$ on $\F$
(or $\G$) do not occur over points of $U$, since for any \w{\bx\in U} and
\w[,]{\xi=(e',e'')\in\cP} the intersection of \w{\bx(e')} and \w{\bx(e'')} is
at a single point internal to both (and in particular, \w{\bx(e')} and
\w{\bx(e'')} are not parallel). Therefore, for any
Cauchy sequence \w{(\bx\sb{i})\sb{i=1}\sp{\infty}} in \w{\Emb(\TG)} (or
\w[)]{\NaG} converging to $\bx$, there is a neighborhood $N$ of $\bx$ in
\w{\Emb(\TG)} (or \w[)]{\NaG} where \w{\phi\sb{\xi}(\bx_{i})} is constant \w{+1}
or constant \w{-1} for all \w[.]{\xi\in\cP}
\end{proof}

We may summarize our constructions so far in the following two diagrams:

%
%
\mydiagram[\label{eqsummaryt}]{
\Emb(\TG) \ar@{^{(}->}[d] \ar@{^{(}->}[rrd]
\ar@{^{(}->}[rr]^{\Phi} &&\F=(\RR{3})\sp{V}\times(\ZZ/3)\sp{\cP}
\ar@{->>}[rd]\sp{q}&\\
\hEmb(\TG) \ar@{->>}[rr]\sp<<<<<<<<<<<{\hPhi} \ar[d]\sp{\widehat{\pi}}&&
\tEmb(\TG) \ar[lld]\sb{\tpi} \ar@{^{(}->}[r] & \hF \ar[llld]^{\hpi}\\
(\RR{3})\sp{V}&&
}
\noindent and similarly for the various types of configuration spaces:
\mydiagram[\label{eqsummaryc}]{
&&\NaG \ar@{^{(}->}[d] \ar@{^{(}->}[rrd]
\ar@{^{(}->}[rr]^{\Phi} &&\G=\CiG\times(\ZZ/3)\sp{\cP}
\ar@{->>}[rd]\sp{q}&\\
&&\CG \ar@{->>}[rr]\sp<<<<<<<<<<<{\hPhi} \ar[d]\sp{\widehat{\pi}}&&
\tCG \ar[lld]\sb{\tpi} \ar@{^{(}->}[r] & \htG \ar[llld]^{\hpi}\\
U \ar@{^{(}->}[r] & \Sigma \ar@{^{(}->}[r]  & \CiG&&
}
\noindent where \w{\widehat{\pi}} is generically a surjection (unless \w{\CG} has
isolated configurations).

%
%
\section{Local description of the blow-up}
\label{clblow}

As we shall see, the global structure of the blow up \w{\tEmb(\TG)}
(or \w[)]{\tCG} can be quite involved, even for the simple linkage \w{\Gamma\sb{2}}
consisting of two lines. The local structure is also hard to understand, in general,
since even the classification of the types of singularities can be arbitrarily
complicated. In the complement of the generic singularity set $U$
(cf.\ \S \ref{sfilt}) we have virtual configurations where:

\begin{enumerate}
\renewcommand{\labelenumi}{(\alph{enumi})~}
\item \w{k\geq 3} edges meet at a single point $P$ ($k$ will be called the
\emph{multiplicity} of the intersection at $P$);
\item Three or more edges meet pairwise (or with higher multiplicities);
\item One or more edges meet at a vertex (not belonging to the edges in question);
\item Two or more vertices (belonging to disjoint sets of edges) meet;
\item Two or more edges coinciding;
\item Any combination of the above situations (including the simple meeting of two
edges at internal points, as in $U$ above) may result in a higher order singularity
if one situation imposes a constraint on another (as when intervening links are
aligned and  stretched to their maximal length).
\end{enumerate}

The goal of this section is to initiate a study of the simpler kinds of singularity
as they appear in the blow-up.

\begin{mysubsection}{Double points}\label{stwoblow}
The simplest non-trivial case of a blow-up occurs for a blown-up configuration
\w{\bw\in\tEmb(\TG)} where two edges \w{e'} and \w{e''} have a single intersection
point $P$ interior to both \w{\bx(e')} and \w[.]{\bx(e'')}
We assume that restricting $\bw$ to the submechanism
\w{\T\sb{\Gamma'}:=\TG\setminus\{e',e''\}} obtained by omitting these two edges
yields an embedding \w[.]{\bw'\in\Emb(\T\sb{\Gamma'})}
In this case we have a neighborhood $N$ of \w{\tpi(\bw)} in \w{(\RR{3})\sp{V}}
for which \w{\tpi\sp{-1}(N)} is a product \w[,]{N'\times M} where \w{N'} is an
open set in a Euclidean space \w{\RR{k}} corresponding to a coordinate patch
around \w{\bw'}
in \w[,]{\Emb(\T\sb{\Gamma'})} while $M$ is diffeomorphic to a suitable open set
in the blow-up \w{\tC(\Gamma\sb{2})} for the two-line mechanism analyzed in
\S \ref{cpgint}A. below.
Thus $M$ is homeomorphic to the disjoint union of two half-spaces:
\w[.]{\HH\sp{8}\times\{\pm 1\}}
\end{mysubsection}

Nevertheless, we can list the simpler types of singularity (outside of $U$).

\begin{mysubsection}{Edge and elbow}\label{sedgeelb}
Now consider the case where an interior point of one edge \w{e\sb{1}} meets a
vertex $v$ common to two other edges \w{e\sb{2}'} and \w{e\sb{2}''} (thus
forming an ``elbow'' \w[),]{\Lambda\sb{2}} as in Figure \ref{eqedgeelbowone}:

%
%
\begin{figure}[ht]
\begin{center}
\begin{tabular}{cccc}
\scalebox{0.4}{\includegraphics{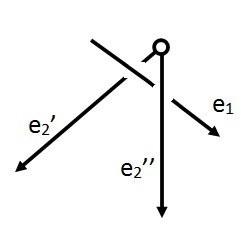}}
\hspace*{-7mm}&
\scalebox{0.4}{\includegraphics{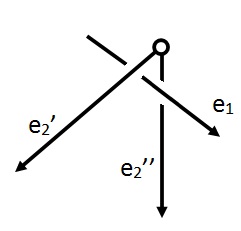}}\hspace*{-7mm}
& \scalebox{0.4}{\includegraphics{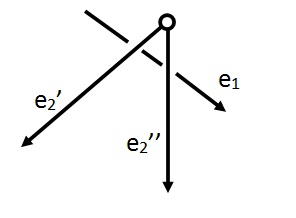}}
\hspace*{-9mm}
& \scalebox{0.35}{\includegraphics{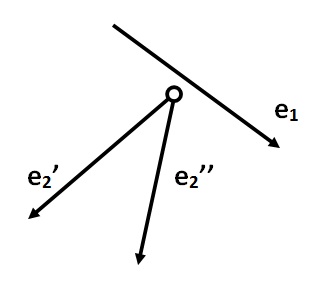}}\\ \\
(a) & (b) & (c)& (d)
\end{tabular}
\end{center}
\caption{Edge and elbow embedding}
\label{eqedgeelbowone}
\end{figure}

We can think of \w{\{e\sb{1},\Lambda\}} as forming a (disconnected) abstract
linkage \w[,]{\TG=(V,E)} so as in \S \ref{dlinknum},
for each embedding \w{\bx:V\to\RR{3}} of \w{\TG} in space we have two
invariants \w{\phi\sb{\xi'}(\bx),\phi\sb{\xi''}(\bx)\in\ZZ/3=\{-1,0,1\}} \wwh
namely, the linking numbers of \w{\bx(e\sb{1})} with \w{\bx(e\sb{2}')}
and of \w{\bx(e\sb{1})} with \w[,]{\bx(e\sb{2}'')} respectively. Together
they yield \w[.]{\vec{\phi}(\bx)\in(\ZZ/3)\sp{\cP}=\ZZ/3\times\ZZ/3}

For example, if in the embedding $\bx$ shown in Figure \ref{eqedgeelbowone}(a)
we have chosen the orientation for \w{\RR{3}} so as to have linking numbers
\w[,]{\vec{\phi}(\bx)=(+1,-1)} say, then
Figure \ref{eqedgeelbowone}(b) will have \w[.]{\vec{\phi}(\bx)=(-1,+1)}

On the other hand, for the embedding of Figure \ref{eqedgeelbowone}(c) we have
\w[,]{\vec{\phi}(\bx)=(-1,-1)}
while for Figure \ref{eqedgeelbowone}(d) we have \w[,]{\vec{\phi}(\bx)=(0,0)}
since the nearest points to \w{\bx(e\sb{1})} on \w{\bx(e\sb{2}')} or
\w{\bx(e\sb{2}'')} are not interior points of the latter.

Thus we see that the immersed configuration represented by Figure
\ref{eqvirtedgeelbowone}(a), in which \w{\bx(e\sb{1})} passes through the vertex
\w[,]{\bx(v)} but is not coplanar with \w{\bx(e\sb{2}')} and \w[,]{\bx(e\sb{2}'')}
has two preimages in the blowup \w[,]{\tEmb(\TG)=\tCG} one of which corresponds to
Figure \ref{eqedgeelbowone}(a) (with invariants \w[),]{(+1,+1)}  while the other
preimage corresponds to both Figures \ref{eqedgeelbowone}(c)-(d), under the
equivalence relation of \S \ref{dblowupr}, since we can have Cauchy sequences of
either type converging to \ref{eqvirtedgeelbowone}(a).

On the other hand, the immersed configuration represented by Figure
\ref{eqvirtedgeelbowone}(b), in which \w{\bx(e\sb{1})} passes through the vertex
\w[,]{\bx(v)} and all edges are coplanar, is represented by three distinct
types of inequivalent Cauchy sequences in \w[,]{\NaG} corresponding to
Figure \ref{eqedgeelbowone}(a), Figure \ref{eqedgeelbowone}(b), and
Figure \ref{eqedgeelbowone}(c)-(d), respectively.  Thus it has
\emph{three} preimages in the blowup \w[.]{\tEmb(\TG)=\tCG}
This is the reason we used invariants in \w[,]{\ZZ/3} rather than \w[.]{\ZZ/2}

%
%
\begin{figure}[ht]
\begin{center}
\begin{tabular}{cc}
\scalebox{0.3}{\includegraphics{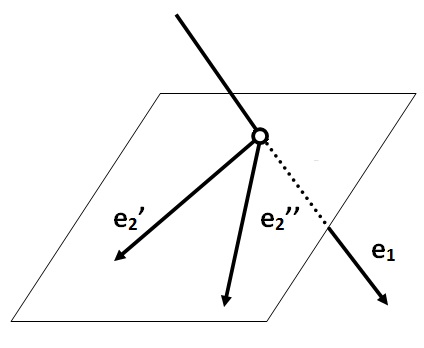}}
\hspace*{30mm}
& \scalebox{0.3}{\includegraphics{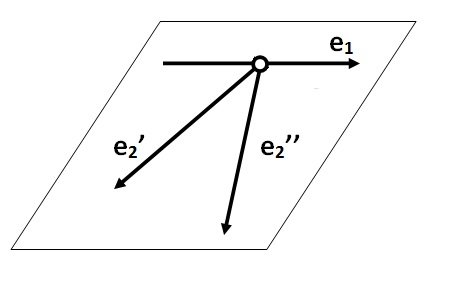}}
\hspace*{20mm}\\
\hspace*{-20mm}(a) & \hspace*{-20mm}(b)
\end{tabular}
\end{center}
\caption{Virtual edge and elbow configuration}
\label{eqvirtedgeelbowone}
\end{figure}

One further situation we must consider in analyzing the edge-elbow linkage
is when two or more edges coincide:

\begin{enumerate}
\renewcommand{\labelenumi}{(\alph{enumi})~}
\item When only \w{\bx(e\sb{2}')} and \w{\bx(e\sb{2}'')} coincide \wh that is,
the elbow is closed \wh we still have the two cases described in Figure
\ref{eqvirtedgeelbowone}.
\item If \w{\bx(e\sb{1})} coincides with \w[,]{\bx(e\sb{2}')} say,
with \w{\bx(v)} internal to \w[,]{\bx(e\sb{1})} the pre-image in
\w{\tEmb(\TG)} is a single virtual configuration (since all cases are
identified under $\sim$). This is true whether or not the elbow is closed.
\end{enumerate}
\end{mysubsection}

\begin{mysubsection}{Two elbows}\label{stwoelb}
Next consider two elbows: \w[,]{\Lambda\sb{1}} consisting of two edges
\w{e\sb{1}'} and \w{e\sb{1}''} with a common vertex \w{v\sb{1}}
and \w[,]{\Lambda\sb{2}} where two other edges \w{e\sb{2}'} and \w{e\sb{2}''}
are joined at the vertex \w[,]{v\sb{2}} as in Figure \ref{eqdoubleelbowone}:

%
%
\begin{figure}[ht]
\begin{center}
\begin{tabular}{ccc}
\scalebox{0.4}{\includegraphics{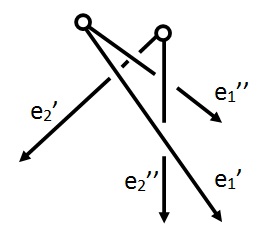}}
\hspace*{5mm}
& \scalebox{0.35}{\includegraphics{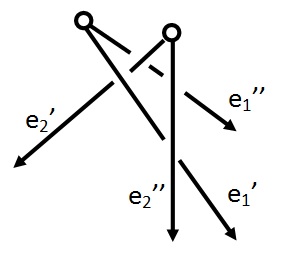}}
\hspace*{5mm}
&\scalebox{0.4}{\includegraphics{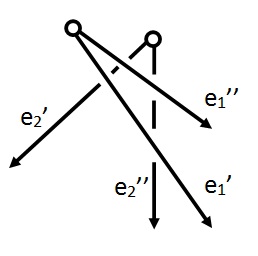}}\\
\hspace*{-25mm}(a) & \hspace*{-25mm}(b) & \hspace*{-25mm}(c)
\end{tabular}
\end{center}
\caption{Double elbow configurations}
\label{eqdoubleelbowone}
\end{figure}

Now we have four two-edge configurations, consisting of pairs of edges
\w[,]{(e\sb{1}',\,e\sb{2}')} \w[,]{(e\sb{1}',\,e\sb{2}'')}
\w{(e\sb{1}'',\,e\sb{2}')} and
\w{(e\sb{1}'',\,e\sb{2}'')} respectively, so \w{\vec{\phi}} takes value in
\w[.]{(\ZZ/3)\sp{\cP}=(\ZZ/3)\sp{4}}

For example, in the embedding of  Figure \ref{eqdoubleelbowone}(a)
we  have \w[,]{\vec{\phi}(\bx):=(+1,+1,+1,-1)}
in Figure \ref{eqdoubleelbowone}(b) we then have
\w[,]{\vec{\phi}(\bx):=(+1,-1,-1,-1)} while in Figure \ref{eqdoubleelbowone}(c)
we have \w[.]{\vec{\phi}(\bx):=(+1,+1,+1,+1)}

The virtual configuration we need to consider is one in which the vertices of the
two elbows coincide. As in Section \ref{sedgeelb}, we must consider a number
of mutual positions in space, as in Figure \ref{eqvirtdoubleelbow}.

%
%
\begin{figure}[ht]
\begin{center}
\begin{tabular}{ccc}
\scalebox{0.3}{\includegraphics{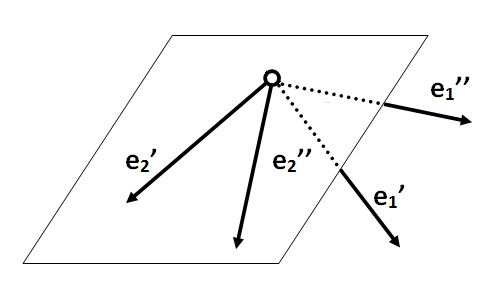}}
\hspace*{15mm}
&\scalebox{0.3}{\includegraphics{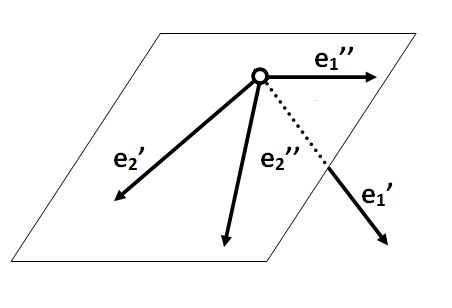}}
\hspace*{15mm}
& \scalebox{0.3}{\includegraphics{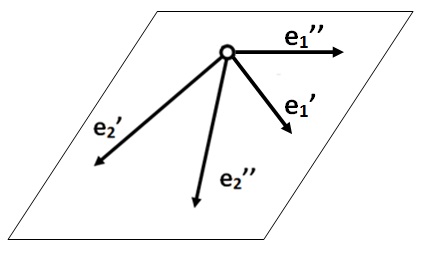}}\\
\hspace*{-10mm}(a) & \hspace*{-10mm}(b) & \hspace*{-10mm}(c)
\end{tabular}
\end{center}
\caption{Virtual double elbow configuration}
\label{eqvirtdoubleelbow}
\end{figure}

Assuming the edges \w{\bx(e\sb{2}')} and \w{\bx(e\sb{2}'')} do not coincide, they
span a plane $E$. If at least one of the edges \w{e\sb{1}'} and \w{e\sb{1}''}
does not lie in $E$, as in Figure \ref{eqvirtdoubleelbow}(a)-(b), we generally
have three possibilities for the blow-up invariants: namely, the limits of the
three cases shown in Figure \ref{eqdoubleelbowone}, where case (c) (and a number
of others) are identified with the case of the two elbows being disjoint.
The same holds if all four edges lie in $E$, as in
Figure \ref{eqvirtdoubleelbow}(c).

On the other hand, if \w{e\sb{1}'} and \w{e\sb{1}''} are on \emph{opposite}
sides of $E$, we cannot have the mutual positions described in Figure
\ref{eqdoubleelbowone}(b), so only two blow-up invariants can occur.

We do not consider here the more complicated cases when one or two of the elbows
are closed, so that the edges \w{e\sb{2}'} and \w[,]{e\sb{2}''} say, coincide
(and thus we have no plane $E$) \wh even though similar considerations may be
applied there.
\end{mysubsection}

%
%
\section{Pairs of generalized intervals}
\label{cpgint}

Even for relatively simple linkage types \w[,]{\TG} the global structure of the
blow up \w{\tEmb(\TG)} (or \w[)]{\tCG} can be quite complicated.  However,
the local structure is more accessible to analysis.

The simplest non-trivial case of a blow-up for a general linkage type \w{\TG}
occurs when two edges \w{e'} and \w{e''} have a single intersection
point $P$ interior to both \w{\bx(e')} and \w[.]{\bx(e'')}
Such a configuration behaves locally like the completed configuration space
of two (generalized) intervals in \w[,]{\RR{3}} which we analyze in this section.

\supsect{\protect{\ref{cpgint}}.A}{Two lines in $\RR{3}$}

We begin with a linkage type \w{\TGn{1}} of two lines \w{\ell\sb{1}} and
\w[.]{\ell\sb{2}} In this case there is no length vector, so \w{\Gn{1}=\TGn{1}} and
\w[.]{\NaGn{1}=\Emb(\TGn{1})} Note that the convention of \S \ref{rlinkage} implies
that each line has a given \emph{orientation}.

For simplicity, we first consider the case where the first line \w{\ell_{1}}
is the (positively oriented) $x$-axis $\vx$, so we need to understand
the choices of the second line \w{\ell=\ell_{2}} in \w[.]{\RR{3}}

Inside the space $\eL$ of oriented lines $\ell$ in \w{\RR{3}} (that is, 
\w[,]{\eL=\NaGn{0}} where \w{\Gamma\sb{0}} consists of a single line) we have the 
subspace \w{\eL\sb{\vx}} of lines intersecting the $x$-axis. We have
\w[,]{\eL\sb{\vx}=(\RR{}\times S^{2})/\sim}
where \w{(x,\vv)\sim(x',\vv')} for any \w{x,x'\in\RR{}} if and only if
\w[,]{\vv,\vv'\in\{\pm\ve\sb{1}\}} with \w{x\mapsto (x,0,0)} sending $\RR{}$ to
\w[.]{\RR{3}} Moreover, we have a map \w{\varphi:\NrGn{1}\to\eL} from the
(unpointed) reduced configuration space \w{\NrGn{1}=\Emr{\TGn{1}}} for
the original linkage, which is a homeomorphism onto \w[,]{\eL\setminus\eL_{\vx}}
since any reduced embedding of \w{\Gn{1}} in \w{\RR{3}} is determined by a choice
of an (oriented) line $\ell$ not intersecting the $x$-axis.

By taking an appropriate $\var$-tubular neighborhood of the two lines,
we may assume that we have two $\var$-cylinders tangent
to each other in \w[,]{\RR{3}} with one of them symmetric about the $x$-axis (see
Figure \ref{ftangent}).
If we denote the space of such tangent cylinders
by \w[,]{X\sb{\var}} we see that \w{\Emr{\TGn{1}}} is a disjoint union
\w[.]{\bigsqcup_{\var>0}\,X_{\var}} Moreover, by re-scaling we see that all the
spaces \w{X_{\var}} are homeomorphic, so in fact
\w[,]{\Emr{\TGn{1}}\cong X_{1}\times(0,\infty)} say.

%
%
\begin{figure}[htb]
\begin{center}
\scalebox{0.2}{\includegraphics{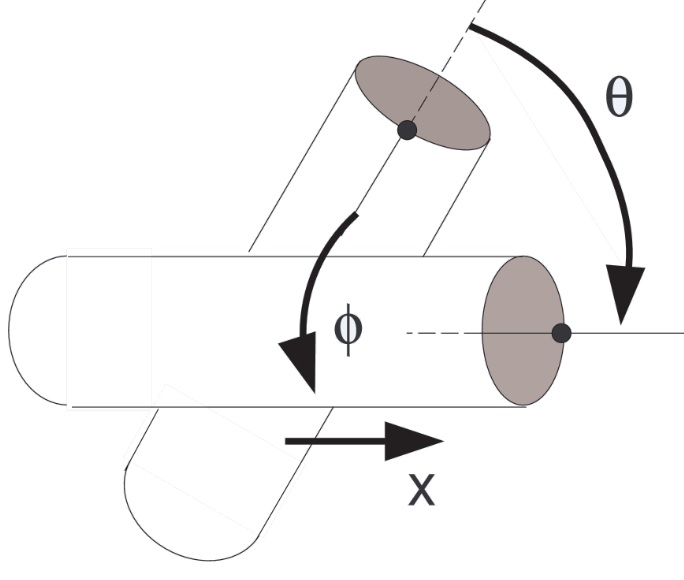}}
\end{center}
\caption{Tangent cylinders}
\label{ftangent}
\end{figure}

Inside the space \w{X\sb{1}} we have a singular locus $\cP$ of configurations
where the two tangent cylinders are parallel. Thus
\w[,]{\cP\cong\bS{1}\times\{\pm 1\}} since such configurations are completely
determined by the rotation $\phi$ of the cylinder \w{T_{1}(\ell)} around $\ell$
with respect to the cylinder \w{T_{1}(\bx)} around the $x$-axis, together with a
choice of the orientation \w{\pm 1} of $\ell$ (relative to $\vx$).

The (open dense) complement \w{X_{1}\setminus\cP} consists of configurations
of a cylinder \w{T_{1}(\ell)} tangent at a single point $Q$ on the boundary
of both cylinders. Such a configuration is determined by:

\begin{enumerate}
\renewcommand{\labelenumi}{(\alph{enumi})~}
\item The projection \w{x} of the point $Q$ on $\vx$;
\item The angle $\theta$ between the (oriented) parallels to the respective axes
$\vx$ and $\ell$ through $Q$; and
\item The rotation $\phi$ of $Q$ about $\vx$ (i.e., the rotation of the
perpendicular \w{\vec{xQ}} to $\vx$ relative to $\vz$).
\end{enumerate}
\noindent (see Figure \ref{ftangent}).

Thus \w{U:=X_{1}\setminus\cP} is diffeomorphic to the open manifold
\w[,]{\RR{}\times\bS{1}\times(\bS{1}\setminus\{0,\pi\})} with global coordinates
\w[,]{(x,\phi,\theta)} identifying $U$ with an open submanifold of the thickened
torus \w[.]{\RR{1}\times\bS{1}\times\bS{1}}
The boundary \w{\RR{1}\times\bS{1}\times\{0,\pi\}} of $U$ is identified in
\w{X_{1}} with $\cP$ by collapsing $\RR{1}$ to a point.

In summary, \w{X_{1}} is homeomorphic to the ``tightened'' torus
\w[,]{(\RR{1}\times\bS{1}\times\bS{1})/\sim} in which two opposite
thickened circles (open annuli) are tightened to ordinary circles
(see Figure \ref{fpinch}).
Thus \w{X_{1}} is homotopy equivalent to a torus \w[.]{\bS{1}\times\bS{1}}

%
%
\begin{figure}[htb]

\begin{center}
\scalebox{0.2}{\includegraphics{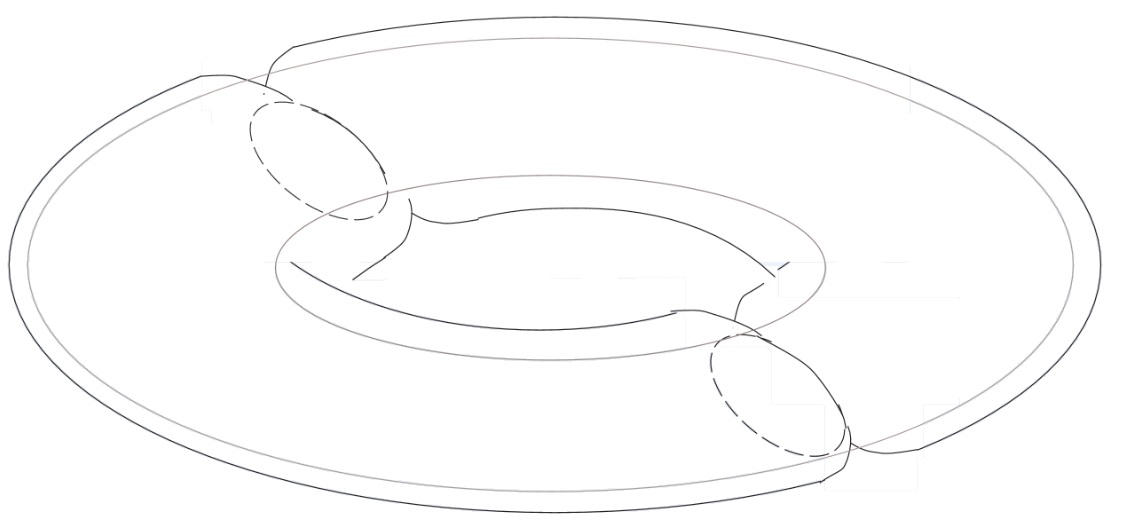}}
\end{center}
\caption{Tightened thickened torus}
\label{fpinch}
\end{figure}

The full space of reduced embeddings \w{\Emr{\TGn{1}}=\NrGn{1}} is still
homotopy equivalent to \w[,]{X_{1}} and thus to \w[.]{\bS{1}\times\bS{1}}
Its completion \w{\hEmr{\TGn{1}}} allows configurations with \w[,]{\var=0} so it
is a quotient of \w[.]{X_{1}\times[0,\infty)} When \w[,]{\theta\neq 0,\pi} we
need no further identifications, since the new ``virtual'' configurations must
still specify on which sides of each other the two lines touch. However, when
$\ell$ is parallel to $\vx$, this has no meaning, so we find:

\begin{myeq}\label{eqconfigone}
\hEmr{\TGn{1}}~=~\CrGn{1}~=~
\left([0,\infty)\times\RR{1}\times\bS{1}\times\bS{1}\right)  /\sim~,
\end{myeq}
\noindent where
\w[:]{(\var_{1},x_{1},\phi_{1},\theta_{1})\sim(\var_{2},x_{2},\phi_{2},\theta_{2})}

\begin{enumerate}
\renewcommand{\labelenumi}{(\alph{enumi})~}
\item for any \w[,]{x_{1},x_{2}\in \RR{1}} if \w[,]{\theta_{1}=\theta_{2}\in\{0,\pi\}}
\w[,]{\phi_{1}=\phi_{2}} and \w[;]{\var_{1}=\var_{2}>0}
\item for any \w{\phi_{1},\phi_{2}\in\bS{1}} and \w[,]{x_{1},x_{2}\in \RR{1}}
if \w{\theta_{1}=\theta_{2}\in\{0,\pi\}} and \w[.]{\var_{1}=\var_{2}=0}
\end{enumerate}

This is because \w{(\var,\phi)} serve as polar coordinates for the plane
perpendicular to $\vx$ through the point of contact $Q$, and \w{\var\cdot\phi}
is the length of the path rotating $\ell$ about $\vx$ when the two are parallel.

We still have a natural map \w{\hat{\varphi}:\CrGn{1}\to\eL} to the space
of oriented lines in \w[,]{\RR{3}} taking a completed configuration to the
location of \w[,]{\ell=\ell_{2}}, but it is two-to-one on all points of
\w{\eL\sb{\vx}} (except $\vx$ itself), since a pair of intersecting lines
corresponds to \emph{two} different completed configurations. To distinguish
between them, we need the following:

\begin{defn}\label{dlinkvec}
As in \S \ref{dlinknum}, given two non-intersecting lines \w{\ell\sb{1}} and
\w{\ell\sb{2}} in \w[,]{\RR{3}} their linking vector
\w{\vw=\vw\sb{(\ell\sb{1},\ell\sb{2})}} is the shortest vector from a
point on \w{\ell\sb{1}} to a point on \w[,]{\ell\sb{2}} perpendicular to their
direction vectors \w{\vu\sb{1}} and \w[.]{\vu\sb{2}}

When the lines are skew, $\vw$ is thus a multiple of \w{\vu\sb{1}\times\vu\sb{2}}
by a scalar \w[,]{a\neq 0} and the linking number
\w{\lk\sb{(\ell\sb{1},\ell\sb{2})}} of the two lines is
\w[,]{\sign(a)\in\{\pm1\}}  with
\w{\lk\sb{(\ell\sb{1},\ell\sb{2})}:=0} when \w{\ell\sb{1}\|\ell\sb{2}} are
parallel, as above.
\end{defn}

By definition, points $\hx$ in the completed configuration space \w{\CrGn{1}}
correspond to (equivalence classes of) Cauchy sequences in the original
space \w[,]{\NrGn{1}} and thus sequences of pairs of $\var$-cylinders tangent
at a common point $Q$. Since we are not interested in configurations over the
special point \w{\vx\in\eL} (for which the two cylinders may be parallel), we
may assume that the angle $\theta$ between the tangents \w{t_{1}} and \w{t_{2}}
to the respective cylinders at $Q$ is not $0$ or $\pi$, so they determine a
plane \w{\spn(t_{1},t_{2})} with a specified normal direction $N$
(towards the second cylinder, having \w{\ell_{2}} as its axis, say).  This is just
the linking vector \w{\vw\sb{(\ell\sb{1},\ell\sb{2})}}  defined above.
The pair \w{(t_{1},t_{2})} converges along the Cauchy sequence to a pair of
intersecting lines \w[,]{(\vx,\ell)} and we define
\w{\hat{\varphi}(\hx):=\ell\in\eL_{\vx}\subseteq\eL} extending the original
homeomorphism \w[.]{\varphi:\NrGn{1}\to\eL\setminus\eL_{\vx}}

Note that any \w{\vx\neq \ell\in\eL_{\vx}} has two sources under
$\hat{\varphi}$, corresponding to the two possible sides of the plane
\w{\spn(\vx,\ell)} on which the cylinder around \w{\ell_{2}} could be.

In fact, \w{\CrGn{1}} is homotopy equivalent to its ``new part''
\w{C:=\CrGn{1}\setminus\NrGn{1}} under the map \w{\psi:\CrGn{1}\to C} sending
\w{[(\var,x,\phi,\theta)]} to \w[.]{[(0,x,\phi,\theta)]} Moreover,
$C$ is homeomorphic to a twice-pinched torus, and so we have shown:

%
%
\begin{prop}\label{ptwoline}
The reduced complete configuration space \w{\CrGn{1}} of two lines in \w{\RR{3}}
is homotopy equivalent to \w[.]{\bS{2}\vee\bS{2}\vee\bS{1}}
Moreover, the map \w{\hat{\varphi}:\CrGn{1}\to\eL} is described up to homotopy
by the map sending each of the two spheres to \w[,]{\bS{2}} with the pinch points 
(corresponding to the two oriented parallels to \w[)]{\ell\sb{1}} sent to the 
two poles.
\end{prop}

For the general case of two oriented lines \w{\ell\sb{1}} and \w{\ell\sb{2}}
in \w[,]{\RR{3}} assume first that in the linkage type \w{\TGn{2}} the line
\w{\ell_{1}} is \emph{framed} \wh that is, equipped with a chosen normal
direction $\vn$. The space of framed oriented lines in \w{\RR{3}} through
the origin is \w[,]{\SO{3}} so the space \w{\eL\sp{\fr}} of all framed
oriented lines is given as for $\eL$ by \w[,]{(\RR{3}\times\SO{3})/\RR{}}
and thus:
\begin{myeq}\label{eqconfigtwo}
\CGn{2}~=~\RR{3}\times\SO{3}\times\bS{1}\times\left([0,\infty)\times
\RR{1}\times\bS{1}\times\bS{1}\right)/\sim~,
\end{myeq}
\noindent where $\sim$ is generated by the equivalence relation of \w{\NrGn{1}}
together with $\RR{}$-translations along the first line \w[.]{\ell_{1}}
To forget the orientations, we divide further by \w[,]{\bS{0}\times\bS{0}} and
to forget the framing, we must further divide by the action of \w{\bS{1}} on $\vn$,
so \w[.]{\CGn{1}=\CGn{2}/\bS{1}}

%
%
\begin{prop}\label{ptwolines}
When \w{\Gamma\sb{1}} consists as above of two (oriented) lines, the completion
\w{\CGn{1}} and the blow-up \w{\tC(\Gamma\sb{1})} of the configuration space 
are homeomorphic.
\end{prop}

\begin{proof}
The virtual configurations in \w{\CGn{1}} are of two types, in which the two lines
intersect or coincide.
In the first case, the linking number of the two lines is the same constant in a
tail of any two Cauchy sequences.
In the second case, any Cauchy sequence is equivalent to one consisting only of
parallel configurations, and therefore the linking numbers must be identified.
\end{proof}

\supsect{\protect{\ref{cpgint}}.B}{Other pairs of generalized intervals}

In principle, all the other pairs of generalized intervals may be treated similarly.
We shall merely point out the changes that need to be made in various special
cases:

\begin{mysubsection}{Line and half-line}\label{slinehalf}
First, consider a linkage \w{\Gn{3}=\TGn{3}} consisting of an oriented line and
a half-line \w[.]{\vp} For the pointed reduced embedding and configuration spaces,
we may assume that \w{\vp} to be the positive direction of the $x$-axis, with
endpoint \w[.]{\bze\in\RR{3}}

We see that \w{\Emsr{\TGn{3}}=\NsrGn{3}} is the subspace of
\w{\Immers^{\red}(\TGn{1})} consisting of oriented lines \w{\ell=\ell_{2}} in
\w{\RR{3}} which do not intersect \w[,]{\vp} so in particular it
is contained in \w[.]{W:=\Immers^{\red}(\TGn{1})\setminus\{\vx\}}
Note that as in \S \ref{cpgint}.A, the lines in $W$ can be globally parameterized by
\w[,]{(\var,x,\phi,\theta)\in[0,\infty)\times\RR{1}\times\bS{1}\times\bS{1}}
module the equivalence relation
\w[.]{\theta\in\{0,\pi\}\Rightarrow \RR{1}\sim\ast} (The case \w{\var=0} allows
the line to intersect $\vx$).

If we let
\begin{myeq}\label{eqlinehalf}
Z~:=~\{[(\var,x,\phi,\theta)]\in
W\subseteq\left([0,\infty)\times\RR{1}\times\bS{1}\times\bS{1}\right)/\sim~:\
\var=0~\Rightarrow~x\leq 0\}
\end{myeq}
\noindent denote the subspace of \w{W\subset\Immers^{\red}(\TGn{1})}
consisting of lines which do not intersect \w[,]{(0,\infty)\subseteq\vp}
we see that the completed configuration space \w{\hEmsr{\TGn{3}}=\CsrGn{3}} is
the pushout:
\mydiagram[\label{eqpushout}]{
\ar @{} [drr]|<<<<<<<<<<<<<<<<<<<<<<<<<<<<<<<<<<<<{\framebox{\scriptsize{PO}}}
\{[(\var,x,\phi,\theta)]\in\NrGn{1}~|\ \theta\not\in\{0,\pi\}\Rightarrow 0<x\}\
\ar@{^{(}->}[d] \ar@{^{(}->}[rr]^<<<<<<<<{f} &&
Z \ar@{^{(}->}[d]  \\
\{[(\var,x,\phi,\theta)]\in\CrGn{1}~|\ \theta\not\in\{0,\pi\}\Rightarrow 0<x\}\
\ar@{^{(}->}[rr] && \CsrGn{3}~,
}
\noindent where $f$ is just the inclusion (using the same coordinates for source
and target).

The quotient map
\w{q_{\Gn{3}}:\CsrGn{3}\to\Immers^{\red}(\TGn{3})=\Immers^{\red}(\TGn{1})}
of \S \ref{rquotient} is induced by the
quotient map \w{q_{\Gn{1}}:\CrGn{1}\to\Immers^{\red}(\TGn{1})} and the inclusion
\w[.]{Z\hra\Immers^{\red}(\TGn{1})}

For the unreduced case, as above we first let \w{\TGn{4}} consist of an oriented
framed pair of line and half-line, so
\w{\CsGn{4}\cong\CsrGn{3}\times\Spin{3}\times\bS{1}} (with the natural basepoint),
and \w{\CsGn{3}} is obtained from \w{\CsGn{4}} by dividing out by
a suitable action of \w[.]{\bS{1}\times\bS{1}}
\end{mysubsection}

\begin{remark}\label{rsegline}
The discussion above carries over essentially unchanged to a linkage
\w{\TGn{5}} consisting of an oriented line \w{L_{1}} and a finite segment $I$:
in the reduced case, we assume \w{I=[0,a]} and then replace the conditions
\w{x\leq 0} and \w{0<x} in \wref{eqlinehalf} and
\wref{eqpushout} with \w{x\not\in(0,a)} and \w[,]{x\in(0,a)}
respectively.

This is the first example where the moduli function
\w{\lambda_{\TG}:\Emb(\TGn{5})\to\RR{1}} is defined. However, all the embedding
configuration spaces \w{\NaGn{5}} are homoeomorphic to each other, and
in fact \w[.]{\Emb(\TGn{5})\cong\NaGn{5}\times(0,\infty)}
\end{remark}

\begin{mysubsection}{Two half-lines}\label{stwohalf}
Let \w{\TGn{6}} be a linkage consisting of two oriented half-lines $\vp$ and $\vq$.
For the reduced pointed space of embeddings \w[,]{\Emsr{\TGn{6}}=\NsrGn{6}} we
now assume as before that \w{\vp} is the positive half of the $x$-axis
$X$, so each point \w{\bx\in\Emsr{\TGn{6}}} is determined by a choice of the vector
$\vv$ from the end point of the $\vq$ to the origin (i.e., the end point of $\vp$),
and the direction vector $\vw$ of the $\vq$. Thus \w{\Emsr{\TGn{6}}}
embeds as \w{(\vv,\vw)} in \w[.]{\RR{3}\times\bS{2}\down{2}}

The only virtual configurations we need to consider are when $\vp$ and $\vv$ are not
collinear, and span a plane $E$ containing $\vp$ and $\vq$. In this case $\vw$
is determined by a single rotation parameter \w{\theta=\theta\sb{\vv}} (relative to
$\vp$), and there are  \w{\theta\sb{0}<\theta\sb{1}} such that $\vq$ intersects
$\vp$ if and only if \w[.]{\theta\sb{0}<\theta\leq\theta\sb{1}} Thus we have an arc
\w{\gamma=\gamma\sb{\vv}} in \w{\bS{2}\down{2}} of values for $\vw$
for which we have two virtual configurations in \w[,]{\hEmb(\TGn{6})} yielding an
embedding of the compactification
\w{\overline{\bS{2}\down{2}\setminus\gamma\sb{\vv}}} of a sphere with a slit
removed in \w[.]{\hEmb(\TGn{6})} Since all such compactifications are
homeomorphic to \w[,]{\overline{\bS{2}\down{2}\setminus\gamma}} say, we see that
$$
\hEmb(\TGn{6})~\cong~X\times \bS{2}\down{2}~\cup
(\RR{3}\setminus X)\times\overline{\bS{2}\down{2}\setminus\gamma}~.
$$
\noindent Since \w{\overline{\bS{2}\down{2}\setminus\gamma}} is contractible, by
collapsing the $x$ axis $X$ to the origin we see that
\w[.]{\hEmb(\TGn{6})\simeq\bS{2}}
\end{mysubsection}

\begin{remark}\label{rbutwoline}
Along the way our analysis showed that for all types of pairs of generalized
intervals the maps \w{\hPhi:\hEmb(\TG)\to\tEmb(\TG)} and \w{\hPhi:\CG\to\tCG}
of \wref{eqsummaryt}-\wref{eqsummaryc} are homeomorphisms \wh that is,
the completed configuration space is identical with the blow-up. This does not hold
for more general linkage types \wh see \S \ref{rinflines} below.
\end{remark}

%
%
\section{Lines in $\RR{3}$}
\label{cdpoints}

In order to deal with more complex virtual configurations, we need to
understand the completed configuration spaces of $n$ lines in \w[.]{\RR{3}}

The configuration space of (non-intersecting) skew lines in \w{\RR{3}} have been
studied from several points of view (see \cite{CPennC,PennC} and the surveys
in \cite{DViroC,ViroT}). However, here we are mainly interested in
the completion of this space, and in particular in the virtual configurations
where the lines ``intersect'' (but still retain the information on their mutual
position, as before).

\begin{defn}\label{dcom}
As above, let \w{\eL=\NaGn{0}} denote the space of oriented lines in \w[,]{\RR{3}} 
and $\nL$ the space of \emph{all} lines in \w{\RR{3}} (so we have a double cover
\w[).]{\eL\epic\nL}

An oriented line is determined by a choice of a basepoint in \w{\RR{3}} and a vector 
in \w[,]{S^{2}} and since the basepoint is immaterial, the space $\eL$
of all oriented lines in \w{\RR{3}} is \w[,]{(\RR{3}\times S^{2})/\RR{}} where 
$\RR{}$ acts by translation of the basepoint along $\ell$.
Alternatively, we can associate to each oriented line $\ell$ the pair
\w[,]{(\vv,\bx)} where \w{\vv\in S\sp{2}} is the unit direction vector of $\ell$,
and $\bx\in\RR{3}$ is the nearest point to the origin on $\ell$,
allowing us to identify:
\begin{myeq}\label{eqorline}
\eL~\cong~\{(\vv,\bx)\in S\sp{2}\times\RR{3}~:\ \bx\cdot\vv=0\}~.
\end{myeq}

If \w{\Gamma\sb{n}=\TGn{n}} is the linkage consisting of $n$ oriented lines, we
thus have an isometric embedding $\pi$ of \w{\NaGn{n}=\Emb(\TGn{n})} in the product
\w[,]{\eLn{n}} which extends to a surjection \w{\hat{\pi}:\hEmb(\TGn{n})\to\eLn{n}}
(no longer one-to-one).

Denote by \w{\eLz{n}} the subspace of \w{\eLn{n}} consisting of all those lines passing through the origin (that is, with \w[),]{\bx=\vze} so
\w[,]{\eLz{n}\cong(S\sp{2})\sp{n}} and let \w[.]{\hE{n}:=\hat{\pi}\sp{-1}(\eLz{n})} Thus \w{\hE{n}} consists of all (necessarily virtual) configurations of $n$ lines passing through the origin.

We denote by $\Sigma$ the subspace of
\w{\eLn{n}\subseteq(S\sp{2}\times\RR{3})\sp{n}} for which at least two of
the $n$ unit vectors in \w{S\sp{2}} are parallel:
$$
\Sigma~:=~\{(\vv\sb{1},\bx\sb{1}\dotsc,\vv\sb{n},\bx\sb{n})\in S\sp{2})\sp{n}~:\
\exists 1\leq i<j\leq n \ \exists 0\neq \lambda\in\RR{}, \ \vv\sb{i}=\lambda\vv\sb{j}\}~.
$$
\noindent Finally, let \w{\hat{\Sigma}:=\hat{\pi}\sp{-1}(\Sigma)} denote the corresponding \emph{singular subspace} of \w[.]{\hEmb(\TGn{n})}
\end{defn}

%
%
\begin{prop}\label{pdefret}
There is a deformation retract \w[.]{\rho:\hEmb(\TGn{n})\to\hE{n}}
\end{prop}

\begin{proof}
For each \w{t\in[0,1]} we may use \wref{eqorline} to define a map
\w{h\sb{t}:\eLn{n}\to\eLn{n}} by setting
$$
h\sb{t}((\vv\sb{1},\bx\sb{1}),\dotsc,(\vv\sb{n},\bx\sb{n}))~:=~
((\vv\sb{1},t\bx\sb{1}),\dotsc,(\vv\sb{n},t\bx\sb{n}))~.
$$
\noindent For \w[,]{t>0} \w{h\sb{t}} is equivalent to applying the $t$-dilitation
about the origin in \w{\RR{3}} to each line in \w[.]{\TGn{n}} Thus it takes the
subspace \w{\Emb(\TGn{n})} of \w{\eLn{n}} to itself, and therefore extends
to a map \w[.]{\widehat{h}\sb{t}:\hEmb(\TGn{n})\to\hEmb(\TGn{n})}

Now consider a Cauchy sequence \w{\{P\up{i}\}_{i=0}\sp{\infty}} in
\w[,]{\Emb(\TGn{n})} of the form
\begin{myeq}\label{eqcauchy}
P\up{i}~=~((\vv\sb{1}\up{i},\bx\sb{1}\up{i}),\dotsc,
(\vv\sb{n}\up{i},\bx\sb{n}\up{i}))~,
\end{myeq}
\noindent converging to a virtual configuration \w[.]{P\in\hEmb(\TGn{n})} Choosing
any sequence \w{(t\sb{i})\sb{i=0}\sp{\infty}} in \w{(0,1]} converging
to $0$, we obtain a new Cauchy sequence
\w{\{h\sb{t\sb{i}}(P\sp{i})\}_{i=0}\sp{\infty}} with
$$
h\sb{t\sb{i}}(P\up{i})~=~((\vv\sb{1}\up{i},t\sb{i}\bx\sb{1}\up{i}),\dotsc,
(\vv\sb{n}\up{i},t\sb{i}\bx\sb{n}\up{i}))~,
$$
\noindent which is still a Cauchy sequence in \w[,]{\Emb(\TGn{n})} and furthermore
\w{\lim_{i\to\infty}~t\sb{i}\bx\sb{j}\up{i}=\bze} for all \w{1\leq j\leq n}
since the vectors \w{(\bx\sb{1}\up{i},\dotsc,\bx\sb{n}\up{i})} have a common bound  $K$ for all \w[.]{i\in\NN} Thus \w{\{h\sb{t\sb{i}}(P\sp{i})\}_{i=0}\sp{\infty}}
represents a virtual configuration \w{P'} in \w{\hEmb(\TGn{n})} with
\w[,]{\pi(P')\in\eLz{n}} and thus \w[.]{P'\in\hE{n}} Moreover, choosing
a different sequence \w{(t\sb{i})\sb{i=0}\sp{\infty}} yields the same \w[.]{P'}
Thus if we set \w[,]{\widehat{h}\sb{0}(P):=P'} we obtain the required map
\w[,]{\rho:=\widehat{h}\sb{0}:\hEmb(\TGn{n})\to\hE{n}} as well as a homotopy
\w{H:\hEmb(\TGn{n})\times[0,1]\to\hEmb(\TGn{n})} with
\w{H(P,t)=\widehat{h}\sb{t}(P)} for \w{t>0} \wwh and thus \w{H(-,1)=\Id} \wwh and
\w[.]{H(-,0)=\rho}
\end{proof}

%
%
\begin{cor}\label{cdefret}
The completed configuration space \w{\hEmb(\TGn{n})} of $n$ oriented lines in
\w{\RR{3}} is homotopy equivalent to the completed space \w{\hE{n}} of $n$
oriented lines through the origin.
\end{cor}

%
%
\section{Three lines in $\RR{3}$}
\label{ctml}

Corollary \ref{cdefret} allows us to reduce the study of the homotopy type of the
completed configuration space of $n$ (oriented) lines in \w{\RR{3}} to the that
of the simpler subspace of $n$ lines through the origin (where we may fix
\w{\ell\sb{1}} to be the $x$-axis).

\begin{mysubsection}{The case of two lines again}\label{stlines}
For \w[,]{n=2} the remaining (oriented) line \w{\ell\sb{2}} is determined by its
direction vector \w[,]{\vv\in\bS{2}} which is aligned with \w{\ell\sb{1}} at the
north pole, say, and reverse-aligned at the south pole. Since we need to take
into account the
linking number \w{\pm1\in\ZZ/2}of \w{\ell\sb{1}} and \w[,]{\ell\sb{2}} we
actually have two copies of \w[.]{\bS{2}} However, the north and south poles
of these spheres, corresponding to the cases when \w{\ell\sb{2}} is aligned
or reverse-aligned with \w[,]{\ell\sb{1}} must be identified as in Figure
\ref{ftwospheres}, so we see that
\w[,]{\hE{2}\simeq\bS{2}\vee\bS{2}\vee\bS{1}} as in
Proposition \ref{ptwoline}.

%
%
\begin{figure}[htb]
\begin{center}
\scalebox{0.4}{\includegraphics{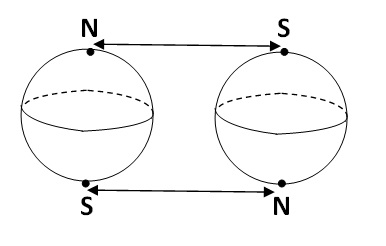}}
\end{center}
\caption{The completed configuration space \ $\hE{2}$}
\label{ftwospheres}
\end{figure}
\end{mysubsection}

\begin{mysubsection}{The cell structure for three lines}
\label{scsthlines}

For \w[,]{n=3} we again fix \w{\ell\sb{1}} to be the positively oriented $x$-axis.
The remaining two lines \w{\ell\sb{2}} and \w{\ell\sb{3}} are determined by
their two direction vectors
\w[.]{(\vv\sb{2},\vv\sb{3})\in\bS{2}\sb{(1)}\times\bS{2}\sb{(2)}}
Again all in all there are eight copies of \w[,]{\bS{2}\times\bS{2}}
indexed by the triples of completed linking numbers
\w{\delta\sb{i,j}:=\phi\sb{(\ell\sb{i},\ell\sb{j})}=\pm 1} for \w{1\leq i<j\leq 3}
(see \S \ref{dblowupr}).

As in \S \ref{stlines}, there are identifications among these products of two
spheres, which occur when at least one pair of
lines is aligned or reverse-aligned: this takes place either in the diagonal
\w{\diag(\bS{2}\sb{(1)}\times\bS{2}\sb{(2)})} (when  \w{\ell\sb{2}} and
\w{\ell\sb{3}} are aligned), in the anti-diagonal
\w{-\diag(\bS{2}\sb{(1)}\times\bS{2}\sb{(2)})} (when  \w{\ell\sb{2}} and
\w{\ell\sb{3}} are reverse-aligned), or in one of the four subspaces of
the form \w{\bS{2}\sb{(1)}\times\{N\}} (when  \w{\ell\sb{1}} and
\w{\ell\sb{2}} are aligned), and so on. Thus we have:
$$
\hE{3}~=~\left(\coprod\sb{1\leq i<j\leq 3}\
[\bS{2}\sb{(1)}\times\bS{2}\sb{(2)}]\sb{\delta\sb{i,j}}\right)/\sim~.
$$

The four special points \w[,]{(N,N),\,(N,S),\,(S,N),\,(S,S)} each appearing in
three of the identification spheres for each of the eight indices
\w[,]{\delta\sb{i,j}} must also be identified, as indicated by the arrows in
Figure \ref{fthreespheres}.

%
%
%
\begin{figure}[htb]
\begin{center}
\scalebox{0.4}{\includegraphics{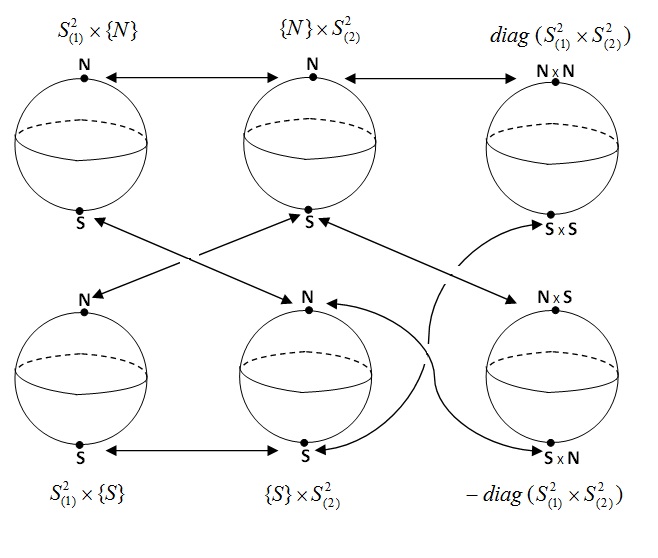}}
\end{center}
\caption{Subspaces of \w{[\bS{2}\sb{(1)}\times\bS{2}\sb{(2)}]\sb{\delta\sb{i,j}}}
where identifications occur}
\label{fthreespheres}
\end{figure}

This suggests the following cell structure for  \w[,]{\hE{3}} in which we
decompose each of the eight copies of \w{\bS{2}\times\bS{2}} into $8$
four-dimensional cells, as follows:

Using cylindrical coordinates \w[,]{(\theta,t)} we think of
\w{\bS{2}} as a cylinder with the top and bottom identified to a point,
(i.e., a square with top and bottom collapsed and vertical sides identified
levelwise). As a result, \w{\bS{2}\times\bS{2}} may be viewed as a
product of the \ww{(t\sb{1},t\sb{2})}-square with the
\ww{(\theta\sb{1},\theta\sb{2})}-square (with suitable identifications), and its
eight cells are obtained by as products of their respective subdivisions,
indicated in Figure \ref{eqpairofpsheres}. Note that it is convenient to replace the
\ww{(\theta\sb{1},\theta\sb{2})}-square by a parallelogram, so the opposite
diagonal edges correspond to \w{\theta\sb{1}=\theta\sb{2}} (identified with
each other). The horizontal edges are also identified pointwise.

%
%
\begin{figure}[ht]
\begin{center}
\begin{tabular}{ccc}
\scalebox{0.3}{\includegraphics{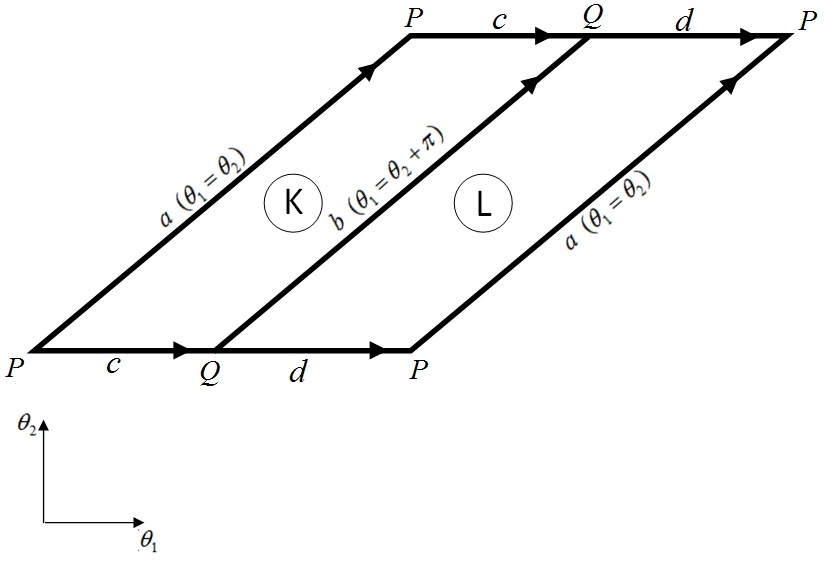}}
&\hspace*{4mm}\raisebox{23mm}{\large $\times$} \hspace*{-10mm}&
\scalebox{0.3}{\includegraphics{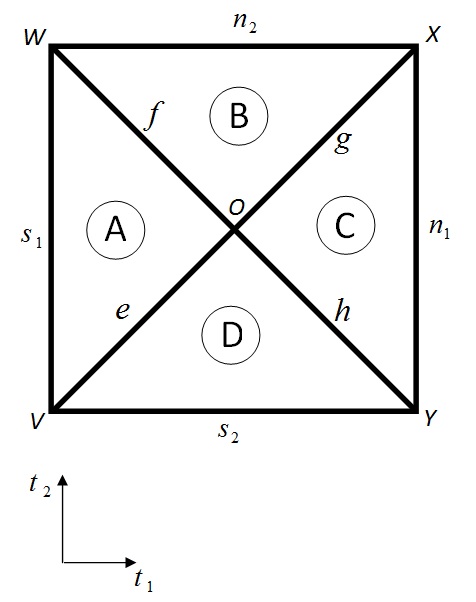}}
\\
\hspace*{-10mm}(a)\ \ $(\theta\sb{1},\theta\sb{2})\in\bS{1}\times\bS{1}$
& &(b) \ \ $(t\sb{1},t\sb{2})\in[-1,1]\times[-1,1]$
\end{tabular}
\end{center}
\caption{$\bS{2}\sb{(1)}\times\bS{2}\sb{(2)}$ \ in double cylindrical
coordinates}
\label{eqpairofpsheres}
\end{figure}

Thus each of the eight copies of \w{\bS{2}\times\bS{2}} decomposes into $8$
$4$-dimensional cells:
\w[.]
{A\times K,B\times K,C\times K, D\times K,A\times L,B\times L,C\times L,D\times L}

However, there are certain collapses in the lower-dimensional products, all
deriving from the fact that when \w{t\sb{i}=\pm1} (at either end of the cylinder),
the variable \w{\theta\sb{i}} has no meaning, so under the quotient map
$$
q~=~q\sb{(1)}\times q\sb{(2)}~:~
\left(\bS{1}\sb{(1)}\times\bS{1}\sb{(2)}\right)\times
\left([-1,1]\sb{(1)}\times[-1,1]\sb{(2)}\right)~
\epic~\bS{2}\sb{(1)}\times\bS{2}\sb{(2)}
$$
\noindent any point \w{(\theta\sb{1},\theta\sb{2},-1,t\sb{2})} is sent to
\w[,]{(S\sb{(1)},q\sb{(2)}(\theta\sb{2},t\sb{2}))} where \w{S\sb{(1)}} is the
south pole in the first sphere \w[,]{\bS{2}\sb{(1)}} and so on. Thus:

\begin{enumerate}
\renewcommand{\labelenumi}{(\arabic{enumi})~}
\item The ostensibly $3$-dimensional cell \w{K\times s\sb{1}} is collapsed
horizontally under $q$ to the $2$-cell \w[.]{S\sb{(1)}\times\bS{2}\sb{(2)}}
Note that the same $2$-cell is also represented by \w{a\times s\sb{1}} and
\w[.]{b\times s\sb{1}}

Similarly \w{q(L\times s\sb{1})=S\sb{(1)}\times\bS{2}\sb{(2)}} and
\w[.]{q(K\times n\sb{1})=q(L\times n\sb{1})= N\sb{(1)}\times\bS{2}\sb{(2)}}
\item On the other hand, \w{K\times s\sb{2}} is collapsed horizontally to the
$2$-cell \w[,]{c\times s\sb{2}} and similarly
\w[,]{q(K\times n\sb{2})=q(c\times n\sb{2})}
\w[,]{q(L\times s\sb{2})=q(d\times s\sb{2})} and
\w[.]{q(L\times n\sb{2})=q(d\times n\sb{2})}
\item The sum \w{(c\cup d)\times s\sb{2}}
is identified under $q$ with \w[,]{\bS{2}\sb{(1)}\times S\sb{(2)}}
which is also represented by \w{a\times s\sb{2}} or \w[.]{b\times s\sb{2}}

Similarly,
\w[.]{q(c\cup d)\times n\sb{2})=
q(a\times s\sb{2})=q(b\times s\sb{2})=\bS{2}\sb{(1)}\times N\sb{(2)}}
\item The two $2$-cells \w{c\times s\sb{1}} and \w{d\times s\sb{1}}
are both collapsed under $q$ to the $1$-cell \w[,]{S\sb{(1)}\times H}
where $H$ is the longitude \w{\theta\sb{2}=0} in \w[.]{\bS{2}\sb{(2)}}
Similarly, \w[.]{q(c\times n\sb{1})=q(d\times n\sb{1})=N\sb{(1)}\times H}
\w[,]{c\times n\sb{1}} and \w{d\times n\sb{1}}
\item Since $V$ corresponds to the pair of south poles \w{(S\sb{(1)},S\sb{(2)})}
in \w[,]{\bS{2}\sb{(1)}\times\bS{2}\sb{(2)}} \w{K\times V} and \w{L\times V}
are collapsed to a single point \w[.]{(S\sb{(1)},S\sb{(2)})}

Similarly,
\w[,]{q(K\times W)=q(L\times W)=(S\sb{(1)},N\sb{(2)})}
\w[,]{q(K\times X)=q(L\times X)=(N\sb{(1)},N\sb{(2)})}  and
\w[.]{q(K\times Y)=q(L\times Y)=(N\sb{(1)},S\sb{(2)})}
\end{enumerate}

In addition, there are identifications among cells associated to the eight $2$-spheres indexed by \w[.]{(\delta\sb{1,2},\,\delta\sb{1,3},\,\delta\sb{2,3})\in(\ZZ/3)^{3}}
These occur only for the $0$-, $1$-, and $2$-cells, when at least two of \w[,]{\ell\sb{1}} \w[,]{\ell\sb{2}} and \w{\ell\sb{3}} are aligned, so \w[.]{\delta\sb{i,j}=0} The resulting identifications
are as follows:

\begin{enumerate}
\renewcommand{\labelenumi}{(\alph{enumi})~}
\item The two $2$-cells \w[,]{a\times e} and \w{a\times g}
consist of pairs of points \w{(t\sb{i},\theta\sb{i})} \wb{i=1,2} with \w{t\sb{1}=t\sb{2}} and \w[,]{\theta\sb{1}=\theta\sb{2}}
so at all points in these cells \w{\ell\sb{2}} and \w{\ell\sb{3}} are aligned. Therefore, \w[,]{\delta\sb{2,3}=0} or equivalently, the corresponding cells in the two products \w{\bS{2}\times\bS{2}} indexed by the triples
\w{(\delta\sb{1,2},\delta\sb{1,3},+1)} and
\w{(\delta\sb{1,2},\delta\sb{1,3},-1)} are identified, for
each of the four choices of \w[.]{(\delta\sb{1,2},\delta\sb{1,3})\in\{\pm1\}\sp{2}}
\item Similarly, \w[,]{b\times f} and \w{b\times h} consist of
pairs of points \w{(t\sb{i},\theta\sb{i})} \wb{i=1,2} with
\w{t\sb{1}=-t\sb{2}} and \w[,]{\theta\sb{1}=\theta\sb{2}}
so \w{\ell\sb{2}} and \w{\ell\sb{3}} are reverse-aligned, and
again the corresponding cells indexed by the triples
\w{(\delta\sb{1,2},\delta\sb{1,3},+1)} and
\w{(\delta\sb{1,2},\delta\sb{1,3},-1)} are identified.
\item The $2$-cell \w{a\times n\sb{1}} (identified with
\w{b\times n\sb{1}} consist of pairs of points with \w[,]{t\sb{1}=+1} so \w{\ell\sb{2}} is aligned with
\w{\ell\sb{1}} and the corresponding cells indexed by the triples \w{(+1,\delta\sb{1,3},\delta\sb{2,3})} and
\w{(-1,\delta\sb{1,3},\delta\sb{2,3})} are identified.
\item The $2$-cell \w{a\times s\sb{1}=b\times s\sb{1}} consist of pairs of points with \w[,]{t\sb{1}=-1} so \w{\ell\sb{2}} is reverse-aligned with
\w{\ell\sb{1}} and the corresponding cells indexed by the triples \w{(+1,\delta\sb{1,3},\delta\sb{2,3})} and
\w{(-1,\delta\sb{1,3},\delta\sb{2,3})} are identified.
\item The two $2$-cells \w[,]{c\times n\sb{2}} and
\w{d\times n\sb{2}} consist of pairs of points with \w[,]{t\sb{2}=+1} so \w{\ell\sb{3}} is aligned with
\w{\ell\sb{1}} and the corresponding cells indexed by the triples \w{(\delta\sb{1,2},+1,\delta\sb{2,3})} and
\w{(\delta\sb{1,2},-1,\delta\sb{2,3})} are identified.
\item The two $2$-cells \w[,]{c\times s\sb{2}} and
\w{d\times s\sb{2}} consist of pairs of points with \w[,]{t\sb{2}=-1} so \w{\ell\sb{3}} is reverse-aligned with
\w{\ell\sb{1}} and the corresponding cells indexed by the triples
\w{(\delta\sb{1,2},+1,\delta\sb{2,3})} and
\w{(\delta\sb{1,3},-1,\delta\sb{2,3})} are identified.
\item From (a) we see that the $1$-cell \w{a\times O}
indexed by\w{(\delta\sb{1,2},\delta\sb{1,3},+1)} and
\w{(\delta\sb{1,2},\delta\sb{1,3},-1)} are identified, and
similarly for \w{P\times e} and \w[.]{P\times g}
\item From (b) we see likewise that the $1$-cells
\w[,]{b\times O} \w[,]{Q\times f} and \w{Q\times h}
indexed by\w{(\delta\sb{1,2},\delta\sb{1,3},+1)} and
\w{(\delta\sb{1,2},\delta\sb{1,3},-1)} are also identified.
\item From (c) and (d) we see that the $1$-cells
\w[,]{P\times s\sb{1}=Q\times s\sb{1}}
\w[,]{P\times n\sb{1}=Q\times n\sb{1}} indexed by
\w{(+1,\delta\sb{1,3},\delta\sb{2,3})} and
\w{(-1,\delta\sb{1,3},\delta\sb{2,3})} are identified.
\item From (e) and (f) we see that the $1$-cells
\w[,]{P\times s\sb{2}} \w[,]{Q\times s\sb{2}}
\w[,]{P\times n\sb{2}} and \w{Q\times n\sb{2}} indexed by
\w{(\delta\sb{1,2},+1,\delta\sb{2,3})} and
\w{(\delta\sb{1,2},-1,\delta\sb{2,3})} are identified.
\item Finally, all six $0$-cells \w[,]{P\times V=Q\times V}
\w[,]{P\times W=Q\times W} \w[,]{P\times X=Q\times X} \w[,]{P\times Y=Q\times Y} \w[,]{P\times O} and \w{Q\times O}
have all three lines \w[,]{\ell\sb{1}} \w[,]{\ell\sb{2}} and
\w{\ell\sb{3}} aligned or reverse-aligned, so
\w{\delta\sb{1,2}=\delta\sb{1,3}=\delta\sb{2,3}=0} and all eight copies are
identified.
\end{enumerate}

Using this cell decomposition, we can easily verify that
\w[,]{H\sb{4}\hE{3}~\cong~\ZZ^{8}}
\noindent where for each of the eight products \w{\bS{2}\times\bS{2}} indexed by
\w{(\delta\sb{1,2},\delta\sb{1,3},\delta\sb{2,3})\in\{\pm 1\}^{3}}
we have a copy of $\ZZ$ generated by the \emph{fundamental $4$-cycle}
\w[.]{\gamma\sb{((\delta\sb{1,2},\delta\sb{1,3},\delta\sb{2,3})}:=
(K+L)\times(A+B+C+D)}
It is also clear that \w{\hE{3}} is connected, so \w[.]{H\sb{0}\hE{3}~\cong~\ZZ}

We leave to the reader to verify that
\w[,]{H\sb{3}(\hE{3};\QQ)~\cong~\QQ\sp{6}} \w[.]{H\sb{2}(\hE{3};\QQ)
~\cong~\QQ\sp{12}} and \w[.]{H\sb{1}(\hE{3};\QQ)~\cong~\QQ\sp{9}}
\end{mysubsection}

%
%
\section{Chains in space}
\label{ccis}

We can use the basic building blocks of Sections \ref{cdpoints}-\ref{ctml}
to study some actual simple linkages, namely, those with a all vertices of valence
$\leq 2$, \ called \emph{chains}.  We begin with the simplest non-trivial
example\vsm:

\supsect{\protect{\ref{ccis}}.A}{Closed quadrilateral chains}

Let $\Gamma$ be a closed quadrilateral chain with vertices $a$,
$b$, $c$, and $d$, and length vector
\w[.]{\vel:=
(\ell\sb{1},\ell\sb{2},\ell\sb{3},\ell\sb{4})=(|ab|,\,|bc|,\,|cd|,\,|ad|)}
See Figure \ref{eqsubmechanism}(b)
We naturally assume the \emph{feasibility inequalities} on $\vel$ (generalized
triangle inequalities), which guarantee that \w{\CG} is non-empty.
In the generic case we have no equations
of the form \w{\ell\sb{1}=\pm\ell\sb{2}\pm\ell\sb{3}\pm\ell\sb{1}}
(which would allow the quadrilateral to be fully aligned).

In the reduced configuration space \w{\NsrG} (cf.\ \S \ref{srestr})
we assume that $a$ is fixed at the origin $\vze$, the link \w{ab} is in the
positive direction of the $x$-axis, so $b$ is fixed at the point
\w[.]{\vb:=(\ell\sb{2},0)} If we mod out by the \ww{\bS{1}}-action
rotating the link \w{ad} in space about the $x$-axis, we obtain the space
\w[,]{\oNsrG} whose points are represented by embeddings of $\Gamma$ for which
the link \w{ad} lies in the closed upper half plane $\eH$ in the
\wwb{x,y}plane.

\begin{mysubsect}{Local description of the singularities in $\CsrG$}
\label{ssingcsrg}

Consider the simpler linkage $\Delta$ consisting of the two links \w{ab} and \w{ad}
of lengths \w{\ell\sb{1}} and \w[,]{\ell\sb{4}} respectively, and with the
distance \w{bd} contained in the closed interval \w{I=[r,R]} for
\w{r:=|\ell\sb{2}-\ell\sb{3}|} and \w[.]{R=\ell\sb{2}+\ell\sb{3}}
See Figure \ref{eqsubmechanism}(a).
The corresponding reduced planar configuration space,
in which we require \w{ab} to lie on the $x$-axis and \w{ad} to lie in $\eH$,
is denoted by \w[,]{\CsrD} and there is a ``forgetful map''
\w[.]{\rho:\oNsrG\to\CsrD}

%
%
\begin{figure}[ht]
\begin{center}
\begin{tabular}{cc}
\scalebox{0.30}{\includegraphics{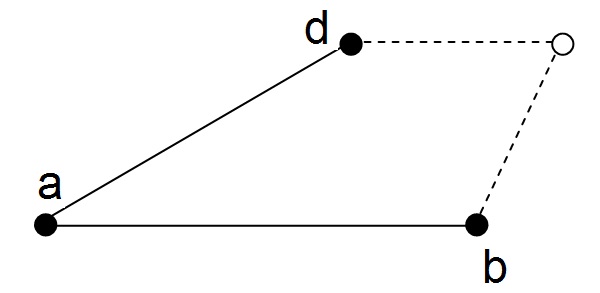}}
\hspace*{5mm}
& \scalebox{0.30}{\includegraphics{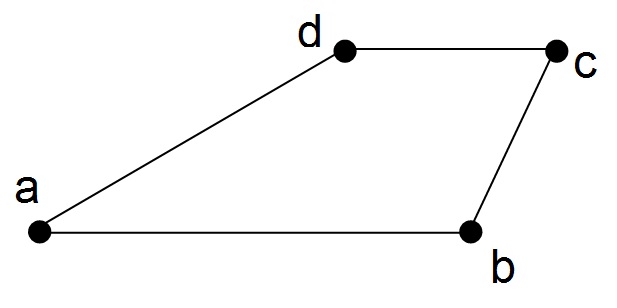}}\\
\hspace*{-55mm}(a) \ The submechanism $\Delta$ & \hspace*{15mm}(b) \
The $4$-chain $\Gamma$
\end{tabular}
\end{center}
\caption{The mechanisms $\Delta$ and $\Gamma$}
\label{eqsubmechanism}
\end{figure}

For a configuration $\bx$ in  \w{\oNsrG} the angle \w{\angle bcd} is determined by
the locations of $b$ and $d$, and thus by the corresponding configuration
\w[,]{\by=\rho(\bx)} and given any \w[,]{\by\in\CsrD}
there is a unique \w{\bbx\in\rho\sp{-1}(\by)} in which $a$ and $c$ do \emph{not} lie
on the same side of the line through \w{bd} (unless $a$, $b$, and $d$ are aligned).

Given this $\bbx$, in \w{\NsrG} the elbow $\Lambda$ formed by $b$, $c$ and $d$
can rotate freely about \w{bd} (unless it is aligned). However, when
we rotate $\Lambda$ from $\bbx$ by \w{180\sp{o}} back into the \wwb{x,y}plane,
the resulting (non-convex) configuration \w{\bbx'} may be self-intersecting,
if the two opposite sides \w{ab} and \w[,]{cd} or else \w{ad} and \w[,]{bc}
intersect in a point interior to one or the other.

By the analysis of planar quadrilateral configurations in \cite[\S 1.3]{FarbT},
we see that given such a self-intersecting planar configuration \w{\bbx'} of
$\Gamma$, one may decrease one angle between adjacent links to
obtain \w{\bx''\in\NsrG} (with angle \w[,]{0\sp{o}} with the two links aligned),
after which the self-intersection disappears.
Therefore, to study the cases in which $\Lambda$ cannot be fully
rotated in \w[,]{\RR{3}} it suffices to consider the configurations
\w{\bx''} where one link is folded onto an adjacent link. We call a case
where \w{bc} is folded back on \w{ab} a \emph{collineation} \w[.]{(acb)}
See Figure \ref{fcollin}.

%
%
\begin{figure}[htb]
\begin{center}
\scalebox{0.2}{\includegraphics{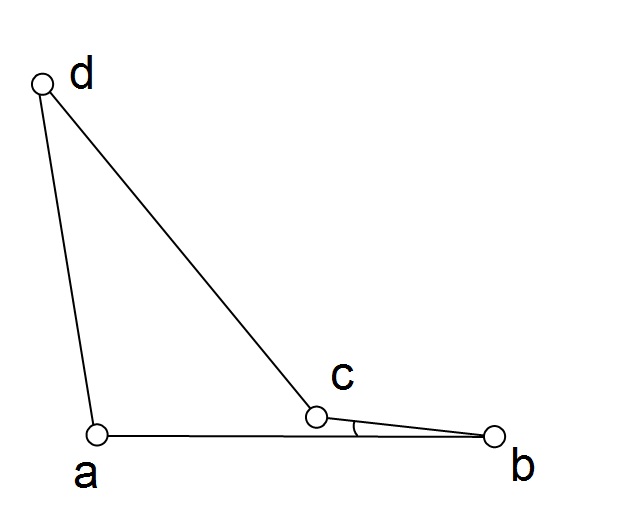}}
\end{center}
\caption{The collineation \ $(acb)$}
\label{fcollin}
\end{figure}

In the full reduced configuration space \w{\NsrG} we have two
angles associated to each configuration \w{\bbx\in\rho\sp{-1}(\by)} as above:
$\theta$ determined by rotating $\Lambda$ about \w[,]{bd} and $\phi$
by rotating the resulting rigid spatial quadrilateral (of \w[)]{\oNsrG} about the
$x$-axis (assuming that $a$, $b$ and $d$ are not collinear).

Note that for all spatial quadrilaterals in \w{\oNsrG} the rotation by $\phi$ is
possible, and yields different spatial configurations of $\Gamma$ (since we are
assuming $\Gamma$ cannot be fully aligned, because $\vel$ is generic). Moreover,
at a collineation \w{\bx''} the full rotation by $\theta$ about \w{bd} is still
allowed.

Any configuration \w{\bx''} representing a collineation \w[,]{(acb)} say,
has a neighborhood in the planar reduced configuration space \w{\oNsrG}
homeomorphic to an open interval \w{(\var,\var)} (with \w{\bx''} itself identified
with the midpoint $0$) such that in one half \w{[0,\var)}
the full rotation by $\phi$ is allowed, while in the other half \w{(-\var,0)}
the rotation by \w{\phi=0\sp{o}} is distinct in the completion (or blowup) \w{\CsrG}
from the rotation by \w[,]{\phi=360\sp{o}} in which the link \w{cd} touches
\w{ab} on opposite sides. Similarly for \w[,]{(acd)} \w[,]{(cab)} and \w[.]{(cad)}

However, when \w{\bx''} represents the collineation \w{(abd)} or \w[,]{(adb)}
the rotation by $\phi$ about the $x$-axis is the same as the rotation by $\theta$
about \w[.]{bd} In this case we simply exchange the roles of \w{bd} and \w{ac}
in defining \w[,]{\oNsrG} and then the same analysis holds. Similarly for
\w{(bdc)} and \w[.]{(dbc)}

Therefore, in the reduced completed spatial configuration space \w{\CsrG}
(in which the only restriction is that side \w{ab} lies on the $x$ axis, with $a$
at the origin) a collinea\-tion configuration \w{\bx''} has a neighborhood $U$
diffeomorphic (as a manifold with corners) to the union of the thickened torus
\w{U\sb{1}:=[0,\var)\times\bS{1}\times\bS{1}} and split thickened torus
\w[,]{U\sb{2}:=(-\var,0]\times [0\sp{o},360\sp{o}]\times\bS{1}} with
\w{(0,\theta,\phi)} in \w{U\sb{1}} identified with \w{(0,\theta,\phi)} in
\w{U\sb{2}}  (and thus in particular \w{(0,0\sp{o},\phi)} and \w{(0,360\sp{o},\phi)}
in \w{U\sb{1}} identified with \w{(0,0\sp{o},\phi)=(0,360\sp{o},\phi)} in
\w[).]{U\sb{2}} The singular
configuration \w{\bx''} is parameterized by the point \w[,]{(0,0\sp{o},0\sp{o})}
while the corresponding convex quadrilateral $\bx$
is parameterized by \w[.]{(0,180\sp{o},0\sp{o})}
\end{mysubsect}

\begin{remark}\label{rmodulifold}
If one link is folded onto an adjacent link, we obtain a triangle with sides
\w[,]{\ell\sb{i}} \w[,]{\ell\sb{j}} and
\w[,]{|\ell\sb{k}-\ell\sb{l}|} respectively
(for \w[).]{\{i,j,k,l\}=\{1,2,3,4\}} For this to be possible, the
three sides must satisfy the triangle inequalities. If we take into
account the feasibility inequalities, these reduce to two cases:

\begin{myeq}\label{eqfirstt}
\ell\sb{i}+\ell\sb{k}>\ell\sb{j}+\ell\sb{l}~,\hsp
\ell\sb{j}+\ell\sb{k}>\ell\sb{i}+\ell\sb{l}~,\hsp\text{and}\hsm \ell\sb{k}>\ell\sb{l}
\end{myeq}
\noindent or

\begin{myeq}\label{eqsecondt}
\ell\sb{j}+\ell\sb{l}>\ell\sb{i}+\ell\sb{k}~,\hsp
\ell\sb{i}+\ell\sb{l}>\ell\sb{j}+\ell\sb{k}~,\hsp\text{and}
\hsm\ell\sb{l}>\ell\sb{k}~.
\end{myeq}
\end{remark}

\begin{mysubsect}{Global description of \ $\NsrG$}
\label{sglobsrg}

As in the classical analysis of \cite{MTrinG}, \w{\CsrD} may be
identified with the intersection \w{\widehat{\be\bh}} of the half-annulus \w{A(d)}
in $\eH$ about $\vb$ of radii \w{r,R} with the half circle \w{B(d)} of radius
\w{\ell\sb{4}} about $\vze$ in the upper half-plane $\eH$ (both
describing possible locations for $d$). See Figure \ref{fredconf}.

%
%
\begin{figure}[htb]
\begin{center}
\scalebox{0.2}{\includegraphics{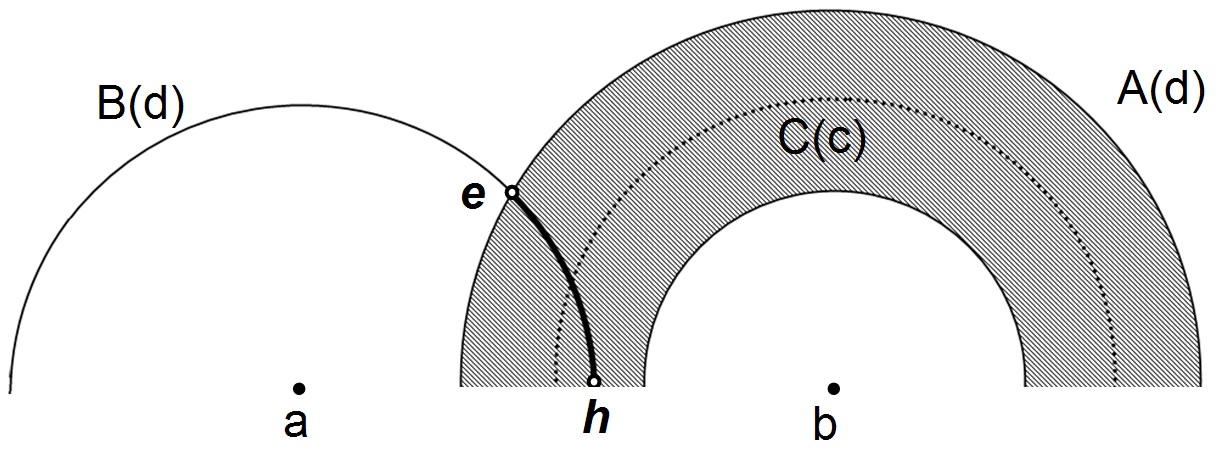}}
\end{center}
\caption{Reduced immersion space \w{\CsrD}}
\label{fredconf}
\end{figure}

We may assume without loss of generality that
\begin{myeq}\label{eqwlog}
\ell\sb{4}<\ell\sb{1}>\ell\sb{2}>\ell\sb{3}~,
\end{myeq}
\noindent so the collineations \w[,]{(dba)} \w[,]{(dbc)} or \w{(cab)} are impossible.

Thus \w{\widehat{\be\bh}} is an arc of \w[,]{B(d)} which is:

\begin{enumerate}
\renewcommand{\labelenumi}{(\roman{enumi})~}
%
\item The full half circle \w{B(d)} when
\begin{myeq}\label{eqfirst}
\ell\sb{1}+\ell\sb{4}<\ell\sb{2}+\ell\sb{3}
\end{myeq}
\noindent (so the leftmost point $\be$ of \w[,]{B(d)} corresponding the links \w{cb}
and \w{cd} being aligned in opposite directions, is in the annulus), and
\begin{myeq}\label{eqsecond}
|\ell\sb{2}-\ell\sb{3}|<\ell\sb{1}-\ell\sb{4}
\end{myeq}
\noindent (so the rightmost point $\bh$ of \w[,]{B(d)} corresponding to the
collineation \w[,]{(bdc)} is in the annulus).
%
\item A proper closed arc of \w{B(d)} ending at the $\bh$ on the positive $x$-axis
when \wref{eqfirst} is reversed and \wref{eqsecond} holds.
%
\item A proper closed arc of \w{B(d)} beginning at $\be$ on the negative $x$-axis
when  \wref{eqfirst} holds and \wref{eqsecond} is reversed.
\item A closed arc of \w{B(d)} not intersecting the $x$-axis when
\wref{eqfirst} and \wref{eqsecond} are reversed (so $\be$ and $\bh$ have positive
$y$-coordinates).
\end{enumerate}

Given a point \w{\by(d)} in this arc, the points $\bbx$ and \w{\bbx'} in
\w{\rho\sp{-1}(\by)\subseteq\oNsrG} are determined by the respective locations
\w{\bbx(c)} and \w{\bbx'(c)} of $c$, which are obtained by intersecting
the circle $E$ of radius \w{\ell\sb{3}} about \w{\by(d)} with the circle
\w{C(c)} of radius \w{\ell\sb{2}} about $\vb$.

Using the analysis in \S \ref{ssingcsrg} we see that a singular point \w{\bx''}
corresponding to the collineation \w{(bdc)} occurs in cases (iii) or (iv) above,
when \w{\bh=\bx''(d)} lies on the inner circle of the half annulus \w[.]{A(d)}
At this \w{\bx''} the rotation by $\theta$ about \w{bd} is trivial.

On the other hand, the collineation \w{(bda)} occurs when \w{\bx''(d)} lies in
the positive half of the $x$-axis, which is possible only in cases (i) or (ii).
The collineation \w{(acb)} occurs at $\bx$ when the circle $G$ of radius
\w{\ell\sb{3}} about the point \w{\bx(c):=(\ell\sb{1}-\ell\sb{2},0)} intersects
\w{\widehat{\be\bh}} at a point $\bbf$ (which is \w[).]{\bx(d)}

To determine when the two remaining mutually exclusive collineations \w{(acd)} or
\w{(cad)} occur, we must interchange the roles of
$d$ and $c$ and study the intersection of the circle \w{C(c)} of radius
\w{\ell\sb{2}} about \w{\vb=(\ell\sb{1},0)} with the inner circle of the annulus
\w{A(c)} about $\vze$ \wh that is with the circle $K$ of radius
\w{|\ell\sb{4}-\ell\sb{3}|} about the origin. Assume that these intersect at the two
points \w[,]{\{\vx\sb{0}(c),\vx\sb{1}(c)\}=K\cap C(c)} with corresponding
lines \w{X\sb{0}} and \w{X\sb{1}} through the origin. The intersections
\w{\bg\sb{0}} and \w{\bg\sb{1}} of these lines with \w{\widehat{\be\bh}}
yield the locations \w{\vx\sb{0}(d)} and \w[.]{\vx\sb{1}(d)}

Thus we can in principle obtain a full description of the
completed configuration space \w[.]{\CG} This will depend on the particular chamber
of the moduli space of all length vectors $\vel$ in \w{\RR{4}\sb{+}} \wwh that is,
which set of inequalities of the form \wref[,]{eqfirstt} \wref[,]{eqsecondt}
\wref[,]{eqfirst}and  \wref{eqsecond} occur (subject to \wref[).]{eqwlog}
\end{mysubsect}

\begin{example}\label{eglongsides}
A particularly simple type of quadrilateral $\Gamma$ is one which has ``three
long sides'' (cf.\ \cite[\S 1]{KMillM}). For instance, if we assume
\w{\ell\sb{1}=\ell\sb{2}=\ell\sb{4}=5} and \w[,]{\ell\sb{3}=1} the arc
\w{\widehat{\be\bh}} is parameterized by the angle \w{\alpha=\angle dab}

When $\alpha$ is maximal, we are at the aligned configuration $\be$, where
the rotation $\theta$ about \w{bd} has no effect, so the fiber of
\w{\rho:\oNsrG\to\CsrD} is a single point (see Figure \ref{eqthreelong}(a)).
Decreasing $\alpha$ slightly as in Figure \ref{eqthreelong}(b) yields a generic
configuration, with fiber \w[.]{\bS{1}} Figure \ref{eqthreelong}(c) represents
the point $\bg$ corresponding to the collineation \w[,]{(acd)} still with fiber
\w[.]{\bS{1}} Further decreasing $\alpha$ yields the self-intersection of
Figure \ref{eqthreelong}(d), with fiber \w[.]{[0\sp{o},360\sp{o}]} This
continues until the minimal $\alpha$ in Figure \ref{eqthreelong}(e), corresponding
to the collineation \w[,]{(bdc)} with trivial fiber.

%
%
\begin{figure}[ht]
\begin{center}
\begin{tabular}{ccccc}
\scalebox{0.12}{\includegraphics{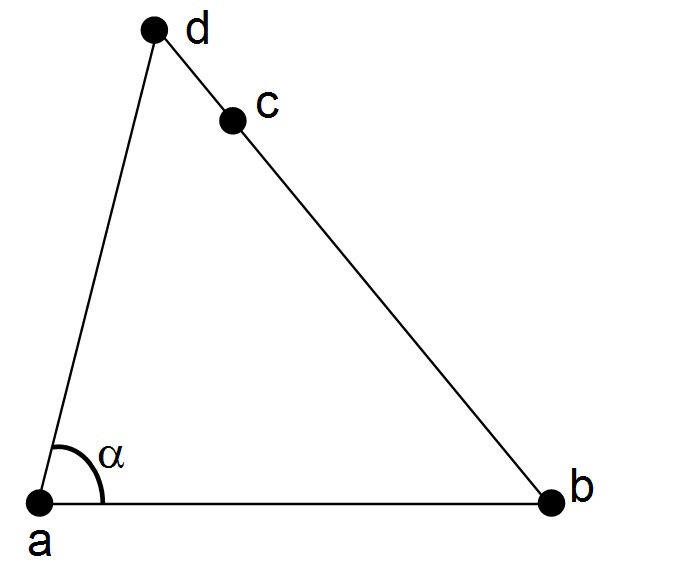}}
\hspace*{5mm}
& \scalebox{0.12}{\includegraphics{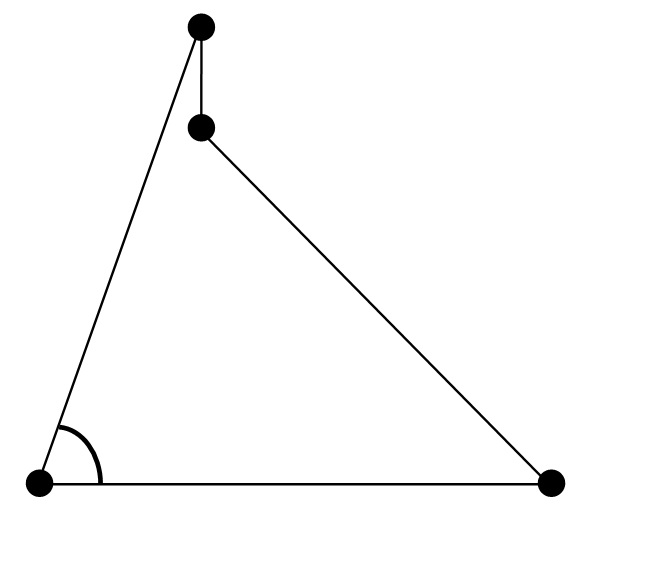}}
\hspace*{5mm}
&  \scalebox{0.12}{\includegraphics{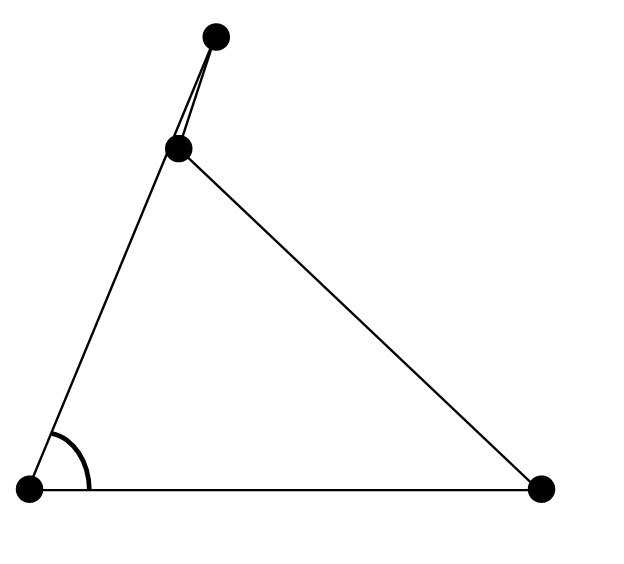}}
\hspace*{5mm}
&  \scalebox{0.12}{\includegraphics{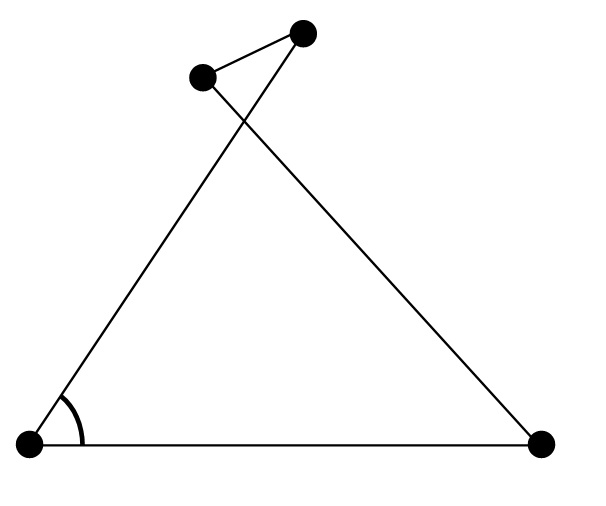}}
\hspace*{5mm}
&  \scalebox{0.12}{\includegraphics{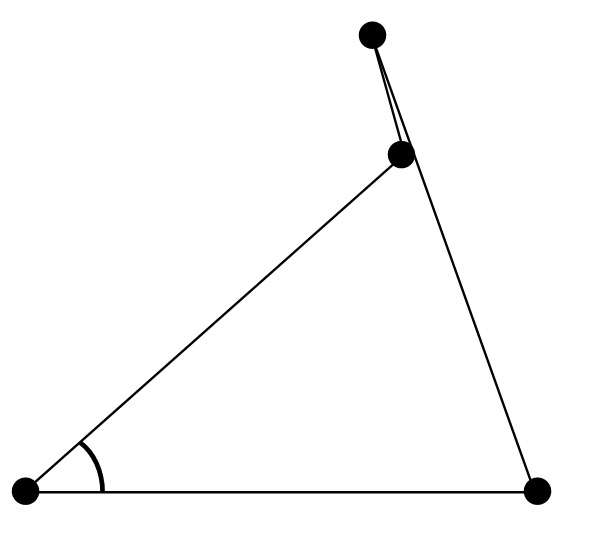}}\\
\hspace*{-15mm}(a) & \hspace*{-20mm}(b) & \hspace*{-20mm}(c) &\hspace*{-20mm}(d) &
\hspace*{-20mm}(e)
\end{tabular}
\end{center}
\caption{A quadrilateral with three long sides}
\label{eqthreelong}
\end{figure}

Thus \w{\oNsrG} is homeomorphic to the slit sphere
\w{\overline{\bS{2}\setminus\widehat{\bg\bh}}} (including the two edges of the cut),
and the full reduced configuration space
\w{\CsrG} is \w[.]{\overline{\bS{2}\setminus\widehat{\bg\bh}}\times\bS{1}}
Finally, the unreduced configuration space is
$$
\w{\CG}~\cong~\overline{\bS{2}\setminus\widehat{\bg\bh}}~\times~
\bS{1}~\times~\bS{2}~\times~\RR{3}~,
$$
\noindent since the quadrilateral can never be full aligned\vsm.
\end{example}

\supsect{\protect{\ref{ccis}}.B}{Open chains}

For an open chain \w{\Gop{2}} with two links (cf.\ \S \ref{egchain}), we have
\w{\Emb{\T_{\Gop{2}}}=\hEmb{\T_{\Gop{2}}}} and \w[,]{\Na{\Gop{2}}=\hC(\Gop{2})} since
no self-intersections exist in our model. Moreover,
\w{\Emsr{\T_{\Gop{2}}}\cong\bS{2}\times(0,\infty)\times(0,\infty)} and
\w[,]{\Nsr{\Gop{2}}\cong\bS{2}} and we can choose spherical coordinates
\w{(\theta,\phi)} for \w[,]{\bS{2}} where $\theta$ is the rotation of the second
edge about the first. The coordinates
\w{(\ell\sb{1},\ell\sb{2})\in(0,\infty)\times(0,\infty)} are the link lengths.

For an open chain \w{\Gop{3}} with three links, we first consider the simplified
case where the first and third link have infinite length: that is, the linkage
\w{\hGop{3}} has two half-lines $\vp$ and $\vq$, whose ends are joined by an
interval $I$.

If we do not specify the length of $I$, this linkage type is equivalent to
\w{\TGn{6}} of \S \ref{slinehalf}. The specific mechanism $\Gamma$ with
\w{|I|=\ell\sp{2}} corresponds to choosing the vector $\vw$ in \S \ref{slinehalf}
to be of length \w[,]{\ell\sp{2}} thus replacing \w{\RR{3}} there by a sphere
\w[,]{\bS{2}} and replacing $X$ by its two poles $N$ and $S$. Thus we see that
$$
\Nr{\hGop{3}}~\simeq~\Nsr{\hGop{3}}~\simeq~\bS{2}\vee\bS{2}\vee\bS{2}~.
$$

The case where \w{\ell\sb{1}} is finite and \w{\ell\sb{3}} is infinite is analogous.
When both are finite, we must distinguish the case when
\w{\ell\sb{1}+\ell\sb{3}>\ell\sb{2}} (again analogous to the infinite case) from
that in which \w{\ell\sb{1}+\ell\sb{3}\leq\ell\sb{2}} (in which case
\w[).]{\Nsr{\hGop{3}}\cong\bS{2}\times\bS{2}}

The analysis of open chains with more links requires a more complicated analysis of
the moduli space of link lengths (see below).

%
%
\section{Appendix: \ Spaces of paths}
\label{aspath}

There is yet another construction which can be used to describe the virtual
configurations of a linkage $\Gamma$, which we include for completeness, even though
it is not used in this paper.

\begin{defn}\label{dequot}
Given a linkage type \w{\TG} with \w[,]{\Emb(\TG)\subseteq(\RR{3})\sp{V}} let
\w{P(\TG)} denote the space of paths \w{\gamma:[0,1]\to(\RR{3})\sp{V}}
such that \w{\gamma((0,1])\subseteq\Emb(\TG)} (cf.\ \cite[Ch.\ 3]{LeeRM}),
and let \w{\eval\sb{0}:P(\TG)\to(\RR{3})\sp{V}} send \w{[\gamma]}
to \w[.]{\gamma(0)}
Let \w{E(\TG)} denote the set of homotopy classes of such paths relative to
\w[.]{\bx=\gamma(0)} This is a quotient space of \w[,]{P(\TG)} with
\w{\oev\sb{0}:E(\TG)\to(\RR{3})\sp{V}} induced by \w[.]{\eval\sb{0}}
We shall call \w{E(\TG)} the \emph{path space of embeddings} of \w[.]{\TG}

Similarly, let \w{P(\Gamma)\subseteq P(\TG)} denote the space of paths
\w{\gamma:[0,1]\to(\RR{3})\sp{V}} such that \w[,]{\gamma((0,1])\subseteq\NaG}
with \w{E(\Gamma)} the corresponding set of relative homotopy classes (a
quotient of \w[).]{P(\Gamma)} We call \w{E(\Gamma)} the \emph{path space of
configurations} of $\Gamma$.
\end{defn}

\begin{lemma}\label{lgerm}
The completed space of embeddings \w{\hEmb(\TG)} is a quotient of the
path space \w[,]{E(\TG)} and \w{\CG} is a quotient of \w[.]{E(\Gamma)}
\end{lemma}

\begin{proof}
First note that when \w[,]{\gamma(0)\in\Emb(\TG)} the path $\gamma$ is completely
contained in the open subspace \w{\Emb(\TG)} of \w[,]{(\RR{3})\sp{V}} so we may
represent any homotopy \w{[\gamma]} by a path contained wholly in an open ball around
\w{\gamma(0)} inside \w[,]{\Emb(\TG)} and any two such paths are linearly
homotopic. Thus \w{\oev\sb{0}} restricted to \w{\Emb(\TG))} is a homeomorphism.

In any completion $\hat{X}$ of a metric space \w[,]{(X,d)} the new points can be
thought of as equivalence classes of Cauchy sequences in $X$.
Since we can extract a Cauchy sequence  (in the path metric) from any path
$\gamma$ as above, and homotopic paths have equivalent Cauchy sequences, this
defines a continuous map  \w[.]{\phi:E(\TG)\to\hEmb(\TG)}

In our case, \w{X=\Emb(\TG)} also has the structure of a manifold, and given any
Cauchy sequence \w{(x\sb{i})\sb{i=1}\sp{\infty}} in $X$, choose an increasing
sequence of integers \w{n\sb{k}} such that \w{\dpath(x\sb{i},x\sb{j})<2\sp{-k}}
for all \w[.]{i,j\geq n\sb{k}} By definition of \w[,]{\dpath} we have a path
\w{\gamma\sp{k}} from \w{x\sb{n\sb{k}}} to \w{x\sb{n\sb{k+1}}} of length
\w[.]{\leq 2\sp{-k}} By concatenating these and using Remark \ref{rsmooth}, we
obtain a smooth path $\gamma$ along which all \w{(x\sb{i})\sb{i=1}\sp{\infty}}
lies. Thus we may restrict attention to Cauchy sequences
\w{(x_{i})_{i=1}^{\infty}=(\gamma(t_{i}))_{i=1}^{\infty}} lying on a smooth path
$\gamma$ in $X$. We can parameterize $\gamma$ so that
\w[,]{\gamma((0,1])\subseteq\Emb(\TG)} and let
\w[,]{\gamma(0):=\lim\sb{i\to\infty}\,x_{i}\in(\RR{3})\sp{V}} which exists since
\w{(\RR{3})\sp{V}} is complete.  Thus $\phi$ is surjective.
\end{proof}

We may thus summarize the constructions in this paper in the following diagram,
generalizing \wref{eqsummaryt}:
\mydiagram[\label{eqsummarytf}]{
&& \Emb(\TG) \ar@{^{(}->}[lld] \ar@{^{(}->}[d] \ar@{^{(}->}[rrd]
\ar@{^{(}->}[rr]^{\Phi} &&\F=(\RR{3})\sp{V}\times(\ZZ/3)\sp{\cP}
\ar@{->>}[rd]\sp{q}&\\
P(\TG)\supseteq  E(\TG) \ar@{->>}[rr]\sp{\phi} \ar[rrd]\sp{\eval\sb{0}}&&
\hEmb(\TG) \ar@{->>}[rr]\sp<<<<<<<<<<<{\hPhi} \ar[d]\sp{\widehat{\pi}}&&
\tEmb(\TG) \ar[lld]\sb{\tpi} \ar@{^{(}->}[r] & \hF \ar[llld]^{\hpi}\\
&& (\RR{3})\sp{V}&&
}
\noindent Similarly for the various types of configuration spaces shown in
\wref[.]{eqsummaryc}

\begin{remark}\label{rinflines}
If we allow a linkage type \w{\T=\TGn{\infty}} consisting of a countable number
of lines, we show that the maps \w{\phi:E(\TG)\to\hEmb(\TG)} and
\w{\phi:E(\Gamma)\to\CG} need not be one-to-one, in general:

Consider the set $S$ of configurations $\bx$ of $\T$ in which the first line
\w{\bx(\ell\sb{0})} is the $x$-axis in \w[,]{\RR{3}} and all other lines
\w{\bx(\ell\sb{n})} \wb{n=1,2,\dotsc} are perpendicular to the \wwb{x,y}plane,
and thus determined by their intersections \w{\bz(\ell\sb{n})} with the
\wwb{x,y}plane, with \w{\bz(\ell\sb{n})=(0,1/n)} for \w[.]{n\geq 2}
Thus the various configurations in $S$ differ only in the location of
\w[.]{\bz(\ell\sb{1})}

Now define the following two Cauchy sequences \w{(\bx'\sb{k})\sb{k=1}\sp{\infty}}
and \w{(\bx''\sb{k})\sb{k=1}\sp{\infty}} in
\w[:]{S\subseteq\Emb(\TGn{\infty})=\C(\Gn{\infty})} we let
\w{\bz'\sb{k}(\ell\sb{1}):=(1/k,1/k)} for \w[,]{\bx'\sb{k}} while
\w{\bz''\sb{k}(\ell\sb{1}):=(-1/k,1/k)} for \w[.]{\bx''\sb{k}} Since we can
define a path in $S$ between \w{\bx'\sb{k}} and \w{\bx''\sb{k}} of length $<3/k$,
the two Cauchy sequences are equivalent, and thus define the same point in the
completions \w[.]{\hEmb(\TGn{\infty})=\CGn{\infty}} However,  if we embed
\w{(\bx'\sb{k})\sb{k=1}\sp{\infty}} in a path \w{\gamma':[0,1]\to S} defined by
the intersection \w[,]{\bz'\sb{t}(\ell\sb{1}):=(t,t)} and similarly for
\w[,]{(\bx''\sb{k})\sb{k=1}\sp{\infty}} it is easy to see that \w{\gamma'} and
\w{\gamma''} cannot be homotopic (relative to endpoints), so they define
different points in \w{E(\TG)} or \w[.]{E(\Gamma)}
\end{remark}

%
%

%
\end{document}